\documentclass[12pt,tbtags,leqno]{amsart}

\usepackage{graphicx}
\usepackage{yfonts}

\usepackage[T1]{fontenc}
\usepackage[latin1]{inputenc}
\usepackage[english]{babel}
\usepackage{amssymb}
\usepackage{amsthm}
\usepackage{amsmath}
\usepackage{verbatim}
\usepackage{hyperref}

\numberwithin{equation}{section}

\newcommand{\HH}{\mathcal{H}}
\newcommand{\HK}{\mathcal{K}}
\newcommand{\HM}{\mathcal{M}}
\newcommand{\HE}{\mathcal{E}}
\newcommand{\HF}{\mathcal{F}}

\newcommand{\HL}{\mathcal{L}}
\newcommand{\HD}{\mathcal{D}}
\newcommand{\HB}{\mathcal{B}}

\newcommand{\Fock}{{\mathcal F^2_d}}
\newcommand{\SymFock}{{\mathcal H^2_d}}
\newcommand{\MFockL}{\mathcal{F}^{\infty, \ell}_{d}}
\newcommand{\MFockR}{\mathcal{F}^{\infty, r}_{d}}
\newcommand{\ncB}{\mathbb{B}^{nc}_d}

\newcommand{\dB}{\partial \mathbb B_d}

\newcommand{\D}{\mathbb{D}}
\newcommand{\B}{\mathbb{B}}
\newcommand{\C}{\mathbb{C}}
\newcommand{\N}{\mathbb{N}}
\newcommand{\R}{\mathbb{R}}
\newcommand{\T}{\mathbb{T}}

\newcommand{\MuH}{\mathrm{Mult(\HH)}}
\newcommand{\Mult}{\mathrm{Mult  }}

\newcommand{\ran}{\mathrm{ran \ }}

\newcommand{\la}{\langle}
\newcommand{\ra}{\rangle}
\newcommand{\Bd}{\mathbb B_d}

\newcommand{\Hol}{\operatorname{Hol}}
\newcommand{\rank}{\operatorname{rank}}
\newcommand{\tr}{\operatorname{tr}}

\newcommand{\clos}{\operatorname{clos}}

\def\HE{{\mathcal E}}
\def\HF{{\mathcal F}}

\def\Bd{{\mathbb{B}_d}}

\newcommand{\Ker}[1]{\mathsf{Ker}~}

\theoremstyle{plain}
\newtheorem{theorem}{Theorem}[section]
\newtheorem{lemma}[theorem]{Lemma}

\newtheorem{prop}[theorem]{Proposition}
\newtheorem{corollary}[theorem]{Corollary}
\newtheorem*{theo}{Theorem}
\newtheorem{definition}[theorem]{Definition}
\newtheorem{ques}[theorem]{Question}
\theoremstyle{definition}
\newtheorem{example}[theorem]{Example}

\author[A. Aleman]{Alexandru Aleman}
\address{Lund University, Mathematics, Faculty of Science, P.O. Box 118, S-221 00 Lund, Sweden}
\email{alexandru.aleman@math.lu.se}

\author[M. Hartz]{Michael Hartz}
\address{Fachrichtung Mathematik, Universit\"at des Saarlandes, 66123 Saarbr\"ucken, Germany}
\email{hartz@math.uni-sb.de}
\thanks{M.H. was partially supported by a GIF grant and by the Emmy Noether Program of the German Research Foundation (DFG Grant 466012782).}

\author[J. M\raise.5ex\hbox{c}Carthy]{John E. M\raise.5ex\hbox{c}Carthy}
\address{Department of Mathematics, Washington University in St. Louis, One Brookings Drive,
	St. Louis, MO 63130, USA}
\email{mccarthy@wustl.edu}
\thanks{J.M. was partially supported by National Science Foundation Grant DMS 2054199.}

\author[S. Richter]{Stefan Richter}
\address{Department of Mathematics, University of Tennessee, 1403 Circle Drive, Knoxville, TN 37996-1320, USA}
\email{srichter@utk.edu}

\begin{document}
\date{\today}
\bibliographystyle{plain}
\title{Free outer functions in complete Pick spaces}
\subjclass[2010]{Primary: 46E22; Secondary: 47A16, 47B32}
\keywords{Reproducing kernel Hilbert space, inner-outer factorization}
\maketitle
\begin{abstract} Jury and Martin establish an analogue of the classical inner-outer factorization of Hardy space functions. They show that every function $f$ in a Hilbert function space with a normalized complete Pick reproducing kernel has a factorization of the type $f=\varphi g$, where $g$ is cyclic, $\varphi$ is a contractive multiplier, and $\|f\|=\|g\|$. In this paper we show that if the cyclic factor is assumed to be what we call free outer, then  the factors are essentially  unique, and we give a characterization of the factors that is intrinsic to the space. That lets us compute examples.  We also provide several   applications of this factorization. \end{abstract}

\tableofcontents
\section{Introduction}
Let $\HH$ be a Hilbert function space on a non-empty set $X$ with reproducing kernel $k$, i.e. $f(x)=\la f, k_x\ra$ holds for every $f\in \HH$ and $x\in X$. We will always assume that Hilbert function spaces are separable.
We write $$\MuH=\{\varphi: X \to \C: \varphi f\in \HH \text{ for all }f\in \HH\}$$ for the multiplier algebra of $\HH$. It is well-known that every multiplier $\varphi$ defines a bounded linear operator $M_\varphi \in \HB(\HH)$ by $M_\varphi f=\varphi f$.  Setting $\|\varphi\|_{\MuH}=\|M_\varphi\|_{\HB(\HH)}$ makes $\MuH$ into a Banach algebra. If $f\in \HH$, then we will write $[f]$ for the multiplier invariant subspace generated by $f$, i.e. the closure of $\Mult(\HH)f$ in $\HH$. A function  $f$ is called cyclic, if $[f]=\HH$.
 Important examples of Hilbert function spaces on the open unit disc $\D$ are the Hardy space $H^2(\D)$ of analytic functions on $\D$ with square summable power series coefficients, the Bergman space $L^2_a$ of square area-integrable analytic functions on $\D$, and the Dirichlet space $D$ of analytic functions $f$ whose derivative $f'\in L^2_a$.

In \cite{ARSBeurling} the classical inner-outer factorization of Hardy space functions was generalized to  the Bergman space setting.

\begin{definition} Let $\HH$ be a separable Hilbert function space on $X$, let $z_0\in X$, and assume that $k_{z_0}=1$.
 A function $f\in \HH$ is called $\HH$-extremal, if $\la \varphi f, f\ra =\varphi(z_0)$ for all $\varphi \in \MuH$.
\end{definition}
It is well-known and easy to check that if $h\in \HH$ with $h(z_0)\ne 0$, then the solutions to
$$\sup\{|f(z_0)|: f \in [h], \|f\|\le 1\}$$ are $\HH$-extremal.
 For $H^2(\D), L^2_a,$ and $D$ we take $z_0=0$. It is clear that the $H^2(\D)$-extremal functions are the classical inner functions, which are those $\varphi \in H^\infty(\D)$ whose radial limit functions satisfy $|\varphi(z)|=1$ for a.e. $|z|=1$. $L^2_a$-extremal functions (resp. $D$-extremal functions) have also been called $L^2_a$-inner  (resp. $D$-inner). The following type of inner-outer factorization was proved for the Bergman space in \cite{ARSBeurling}.

\begin{theo}\label{BergmanBeurling} For each non-zero $f\in L^2_a$ there is a unique pair of functions $\varphi, h\in L^2_a$ such that
\begin{enumerate}
\item $f=\varphi h $
\item $h$ is cyclic in $L^2_a$ and $h(0)>0$,
\item $\varphi$ is $L^2_a$-extremal,
\item $[f]=[\varphi]$.
\end{enumerate}
Furthermore, $f/\varphi \in L^2_a$ and  the contractive divisor property  $\|p\frac{f}{\varphi}\|\le \|pf\|$ holds for all polynomials $p$.
\end{theo}
  Borichev and Hedenmalm proved that conditions (i),(ii), and (iii)  in general do not determine $\varphi$ and $h$ uniquely, \cite{BoriHe}. For more information about Bergman spaces we refer the reader to the monographs \cite{HeKoZhu} and \cite{DuSchu}.

Theorems similar to the above Theorem are known to hold for certain weighted Bergman spaces, see \cite{Shi_alpha}, \cite{HeJaSh}, \cite{McCRi}. For the Dirichlet space $D$ extremal functions $\varphi$ were shown to have several nice properties: their derivatives are in the Nevanlinna class, $(z\varphi)'$ has a meromorphic pseudocontinuation in the exterior disc, and $\varphi$ is a contractive multiplier, \cite{RiSuTAMS}. In particular, by use of the contractive multiplier property, it is elementary to show that in the case of the Dirichlet space conditions (i),(ii), and (iii) of the above Theorem  imply condition (iv). However, in Theorem \ref{DiriNotFactor} we will show that there are functions $f$ in $D$ that are not of the form $f= \varphi h$ for $\varphi, h\in D$ with $\varphi$ $D$-extremal and $h$ cyclic in $D$.  In \cite{JuMa1} and \cite{JuMa2} Jury and Martin observe that results of Arias-Popescu and Davidson-Pitts (\cite{AriasPopescu}, \cite{DavidsonPitts}) lead to a factorization of Dirichlet functions of the type as in the above Theorem (i)-(iii), except that instead of extremal functions one has to allow a larger class of contractive multipliers $\varphi$. Jury and Martin started by proving results for the Drury-Arveson space $H^2_d$, but then they also establish  their results for all spaces with complete Pick kernel and this includes the Dirichlet space (definitions below).

In this paper we will work in the context of all spaces with complete Pick kernels, and we will call the Jury-Martin contractive multipliers subinner. Furthermore, we will show that one obtains a unique factorization, if one restricts the cyclic factors to a class of functions that we will call the free outer functions. We will also characterize the subinner and the free outer functions in terms of the space $\HH$. This will allow us to compute examples and  give some applications.

From now on let $\HH$ be a separable Hilbert function space on a set $X$ whose reproducing kernel $k$ is a   complete Pick kernel  that is normalized at some point $z_0\in X$. That means that $k_w(z)= \frac{1}{1-\la u(z),u(w)\ra_{\HK}}$, where $u$ is a function from $X$ into the open unit ball of an auxiliary separable Hilbert space $\HK$ with  $u(z_0)=0$.  Important examples of such spaces on the unit disc $\D$ are the Hardy space $H^2(\D)$, where $k_w(z)=\frac{1}{1-\overline{w}z}$, and the  Dirichlet space $D$ with reproducing kernel $k_w(z) = \frac{1}{\overline{w}z}\log \frac{1}{1-\overline{w}z}$ (\cite{AgMcC}, Corollary 7.41). The Bergman space $L^2_a$  has reproducing kernel $k_w(z)=\frac{1}{(1-\overline{w}z)^2}$, which is not a complete Pick kernel. Examples of spaces of multivariable functions  with complete Pick kernels are the Drury-Arveson space $H^2_d$ of analytic functions on the unit ball $\B_d$ of $\C^d$, if $d\in \N$, $k_w(z)=\frac{1}{1-\la z, w\ra_{\C^d}}$, and  $H^2_\infty$ the space of functions on the unit ball $\mathbb{B}_\infty$ of $\ell^2$ with reproducing kernel $k_w(z)=\frac{1}{1-\la z, w\ra_{\ell^2}}$. In the following, when we refer to $H^2_d$ we will usually mean to refer to all cases $d\in \N \cup \{\infty\}$.

It is known that it follows from the complete Pick property that $\Mult(\HH)$ has a predual with the property that point evaluations on $\Mult(\HH)$ are weak*-continuous and that it has property $\mathbb{A}_1(1)$, i.e. for every weak*-continuous linear functional $L$ on $\HB(\HH)$ with norm $\|L\|<1$, there are $f,g\in \HH$ such that $\|f\|\|g\|\le 1$ and $L(M_\varphi)=\la \varphi f,g\ra$ for all $\varphi \in \MuH$, see \cite{DavidsonHamilton}, and also see Section \ref{SecVectorvalued} for a proof phrased in terms of the concepts of this paper. We say that a weak*-continuous linear functional on $\MuH$ is a {\it  vector state}, if it is of the form $P_f(\varphi)=\la \varphi f,f\ra$
for some $f \in \mathcal{H}$.

 It is clear that it may happen that $P_f=P_g$ even for $f\ne g$. For example, if $\HH=H^2(\D)$, then one easily checks that $P_f=P_g$ if and only if $|f|=|g|$ a.e. on $\partial \D$, i.e.\ if and only if $f$ and $g$
 have the same outer factor.  We will use these vector states to partition $\HH$, i.e. for $f\in \HH$ we set
 $$\HE_f=\{g\in \HH: P_f=P_g\}.$$

 Note that because of the normalization at the point $z_0\in X$ we have  $k_{z_0}=1$ and hence $P_1(\varphi)=\varphi(z_0)$. Thus, we observe $$\HE_1=\{f \in \HH: f \text{ is }\HH \text{-extremal}\}.$$ Hence if $\HH=H^2(\D)$, then $\HE_1$ equals the set of the classical inner functions.

 \begin{definition} \label{subinnerDef} (a) A function $f\in \HH \setminus \{0\}$ is called {\it free outer}, if $$|f(z_0)|=\sup\{|g(z_0)|: g \in \HE_f\}.$$
(b) A multiplier $\varphi\in \Mult(\HH)$ is called {\it subinner}, if $\|\varphi\|_{\Mult(\HH)}=1$ and if there is $h\in \HH$, $h\ne 0$ with $\|\varphi h\|=\|h\|$.

(c) A pair $(\varphi,f)$ is called a subinner/free outer pair, if $\varphi$ is subinner, $f$ is free outer with $f(z_0)>0$, and $\|\varphi f\|=\|f\|$.
\end{definition}

If $\HH=H^2(\D)$, then one easily checks that the free outer functions are just the outer functions and the subinner functions are the classical inner functions. In the generality of all normalized complete Pick spaces it is known that extremal functions $f\in \HH$ are multipliers of unit norm, i.e $\|f\|_{\MuH}=\|f\|_\HH=1$. If $f(z_0)\ne 0$, then this can be derived from the main result of \cite{McCulloughTrent}. For the general case and an explicit proof see Corollary 4.2 of \cite{AHMcCR_Factor}, and it also follows from our Theorem \ref{subinner/free outer}.

It follows that every extremal function is subinner, but we will see that in  general the collection of subinner functions is strictly larger than the collection of extremal functions. On the other hand it is clear that if $\varphi$ is any single variable $H^\infty$ function with $\|\varphi\|_\infty =1$, but that is not inner, then $\varphi$ is a function of multiplier norm 1 that is not subinner (e.g. $\varphi(z)=(1+z)/2$). The same single variable function then also defines a multiplier $M_\varphi$ of norm 1 on $H^2_d$ for all $d \ge 1$, and one checks that $M_\varphi$ is not subinner on $H^2_d$.

In Corollary \ref{subinner has free outer} we will show that subinner functions $\varphi$ are what has been called {\it column extreme}, i.e.  there is no nonzero multiplier $\psi$ such that the column multiplier $h \mapsto \left(\begin{matrix} \varphi h\\ \psi h\end{matrix}\right)$ is contractive $\HH\to \HH\oplus \HH$. Furthermore, as  has been noted by Jury and Martin,  column extreme multipliers are extreme points of the unit ball of $\Mult(\HH)$, see the proof of Corollary 1.2 of \cite{JuMa3}. Thus, all subinner functions in complete Pick spaces are extreme points of the unit ball of $\Mult(\HH)$. In Section \ref{Sec:SubinnerDirichlet} we will present a condition which implies that in the Dirichlet space a function that is sufficiently regular in $\overline{\D}$ is subinner, if and only if it has multiplier norm 1, see Theorem \ref{DiriSubinner}. We will prove a similar theorem for certain weighted Dirichlet spaces on $\Bd$, but we note that the example $\varphi(z)=(1+z_1)/2$ shows that no such theorem can hold in $H^2_d$.

The important connection between subinner functions and the partition used to define free outer functions becomes apparent by use of an elementary Lemma, which we state explicitly, so that we can refer to it later.
It is a direct consequence of the equality case in the Cauchy--Schwarz inequality.

\begin{lemma}\label{contractionLemma} Let $\HK, \mathcal{L}$ be Hilbert spaces and $T\in \HB(\HK,\mathcal{L})$ with $\|T\|\le 1$. If $x\in \HK$ with $\|Tx\|=\|x\|$, then $T^*Tx=x.$
\end{lemma}

 We apply this with $T=M_\varphi$ for a subinner function $\varphi \in \MuH$ and an associated   $h\in \HH, h\ne 0$ with $\|\varphi h\|=\|h\|$. Then $M^*_\varphi M_\varphi h=h$, and hence for all $\psi \in \MuH$ we have $$P_{\varphi h}(\psi)= \la \psi \varphi h, \varphi h\ra = \la \psi h, M^*_\varphi M_\varphi h\ra = P_h(\psi).$$ Thus $\HE_{\varphi h}=\HE_h$.

It is not immediately clear that free outer functions exist, much less that one is contained in each set $\HE_f$. But that will be one of our main results.
\begin{theorem}\label{subinner/free outer} Let $\HH$ be a Hilbert function space with normalized complete Pick kernel. Then for every $f\in \HH \setminus\{0\}$ there is a unique  subinner/free outer pair $(\varphi,h)$ such that $f=\varphi h$. \end{theorem}
 We will prove this in Section \ref{SecPickspaces}. In the following we will  refer to $h$ as the free outer factor of $f$ and $\varphi$ as the subinner factor of $f$. In Theorem \ref{VectorFactorization} we will provide a vector version of this theorem. It will imply that any square summable sequence of functions in $\HH$ has a common free outer factor.

 As mentioned before, the factorization is the same as the one given by Jury and Martin in \cite{JuMa1,JuMa2}. We will give complete details later, but it is known that there is a natural embedding of $\HH$ in the free Fock space $\Fock$ for some $d\in \N \cup \{\infty\}$, and that for functions in $\Fock$ there is a free inner-outer factorization (\cite{AriasPopescu}, \cite{DavidsonPitts}). It is thus clear that this free inner-outer factorization applies to functions in the embedded spaces $\HH\subseteq \Fock$. The factorization of Theorem \ref{subinner/free outer} is just that factorization from the free setting, and the novelty here is the uniqueness and the observation that the factors can be described intrinsically in terms of the space $\HH$. We will see in Theorem \ref{uniqueLiftofSubinner} that a multiplier is subinner, if and only if it has a unique lift to an isometric left multiplier on $\Fock$. Furthermore, a function in $\HH$ is free outer, if and only if its image under the embedding in $\Fock$ is left- (respectively right-) outer (see Theorem \ref{free-outer-in*invariant} and Lemma \ref{lem:free_outer_H}). We also note that the defining property of free outer functions is inspired by Theorem 2.3 of \cite{GeluPopMultiA_I} and Proposition 2.5 of \cite{DavidsonPitts}.

 By the observation before Theorem \ref{subinner/free outer}, if $f=\varphi h$ is a subinner/free outer factorization, then $\HE_f=\HE_h$, and it follows that each equivalence class $\HE_f$ contains a free outer function. It turns out to be unique up to a constant factor of modulus 1.
\begin{theorem} \label{Pf=Pg} Let $\HH$ be a Hilbert function space with normalized complete Pick kernel.
If $f ,g\in \HH \setminus \{0\}$, then $P_f=P_g$ if and only if $f$ and $g$ have the same free outer factors.\end{theorem}

This will be proved in Section \ref{SecPickspaces}. We will see that free outer functions are always cyclic vectors in $\HH$ (see Theorem  \ref{freeOuterIsCyclic}). This easily implies that $[f]=[\varphi]$, whenever $\varphi$ is the subinner factor of $f$. However, we will show by example that in the Dirichlet space $D$ and in the Drury-Arveson space $H^2_d$, $d\ge 2$, there are  cyclic functions that are not free outer (Examples  \ref{ExampleDirichlet} and \ref{ExampleDruryArveson}). This implies that in these spaces there are cyclic, non-constant subinner functions. For the Dirichlet space we will even exhibit non-constant functions that are simultaneously subinner and free outer (see the remark following Theorem \ref{DiriSubinner}).
\begin{corollary}\label{freeouter} Let $h\in \HH \setminus \{0\}$. Then the following are equivalent:
\begin{enumerate}
\item[(a)]$h$ is free outer,
\item[(b)] there is $z\in X$ with $|h(z)|=\sup\{|f(z)|: f\in \HE_h\}$,
\item[(c)] for all $z\in X$ we have $|h(z)|=\sup\{|f(z)|: f\in \HE_h\}$.
\end{enumerate}
\noindent If $\HH$ does not admit any non-constant isometric multipliers, then the three conditions above are also equivalent to
\begin{enumerate}
\item[(d)] $\|\psi f\|\le \|\psi h\|$ for all $f\in \HE_h$ and $\psi \in \MuH$,
\end{enumerate}
\end{corollary}
We note that many spaces including the Dirichlet space $D$ and the Drury-Arveson space $H^2_d$ for $d\ge 2$ satisfy the hypothesis of item (d) of this corollary.
For $H^2_d$, this is Proposition 8.36 in \cite{AgMcC}, and the statement about $D$ also follows from a computation
with power series.
\begin{proof}
  The previous two Theorems imply that if $h$ is free outer, then each $f \in \mathcal{E}_h$ has the
  form $f = \varphi h$ for a contractive multiplier $\varphi$. This shows the implications
  (a) $\Rightarrow$ (c) and (a) $\Rightarrow$ (d). The implication (c) $\Rightarrow$ (b) is trivial. The implication (b) $\Rightarrow$ (a) follows from the existence of the subinner/free outer factorization and the fact that a non-constant contractive multiplier $\varphi$ must satisfy $|\varphi(z)|<1$ for each $z\in X$, see Lemma 2.2 of \cite{AHMcCR_Smirnov}.

We will now prove (d) $\Rightarrow$ (a). Suppose $h\in \HH$ satisfies that $\|\psi f\|\le \|\psi h\|$ for all $f \in \HE_h$ and all $\psi \in \Mult(\HH)$. Let $h=\varphi g$ be the subinner/free outer factorization of $h$, then for all $\psi\in \Mult(\HH)$ we have
$$\|\psi g\|\le \|\psi h\|=\|\psi \varphi g\|\le \|\psi g\|.$$
Thus $\|\varphi (\psi g)\|=\|\psi g\|$ for all $\psi\in \Mult(\HH)$. Now since the free outer function $g$ is cyclic in $\HH$ it follows that $M_\varphi$ is an isometry on $\HH$, and hence by hypothesis $\varphi$ must be constant, hence $h$ is free outer.
\end{proof}

\begin{corollary} \label{subinner has free outer} If $\varphi \in \Mult(\HH)$ is subinner, then there is a free outer function $f$ with $\|\varphi f\|=\|f\|\ne 0$. Consequently, subinner functions are column extreme multipliers.\end{corollary}
\begin{proof} Since $\varphi$ is subinner, there is a non-zero $g\in \HH$ with $\|\varphi g\|=\|g\|$. Let $g=\psi f$ be the subinner/free outer factorization of $g$, then $\|f\|=\|\psi f\|=\|\varphi \psi f\|\le \|\varphi f\|\le \|f\|$. Hence $\|\varphi f\|=\|f\|\ne 0$ and $f$ is free outer.

Since free outer functions are cyclic we have $f(z)\ne 0$ for all $z$, hence if $\psi \in \Mult(\HH)$ such that $h \mapsto \left(\begin{matrix}\varphi h\\ \psi h\end{matrix}\right)$ is contractive $\HH \to \HH \oplus \HH$, then $\|\psi f\|=0$ and hence $\psi=0$. Thus, $\varphi$ is column extreme.
\end{proof}

In Sections \ref{SecDiriExample} and \ref{SectionDruryArveson} we have calculated some examples of subinner/free outer factorizations in the Dirichlet and Drury-Arveson spaces. We now present a general observation about examples. Note that if $f=\varphi h$ is a subinner/free outer factorization, then $M_\varphi^* f=h$ by Lemma \ref{contractionLemma}, and hence the free outer factor of $f$ is contained in the $*$-invariant subspace generated by $f$: $$h\in [f]_*=\clos \{M_\psi^* f: \psi \in \MuH\}.$$ So for example, if $f$ is a finite linear combination of reproducing kernels, then its free outer factor is a linear combination of the same reproducing kernels and hence the subinner factor is a ratio of two such expressions. In particular, we note that reproducing kernels must be free outer.

\begin{example} If $\HM$ is any multiplier invariant subspace of $\HH$ and $\lambda$ is not a common zero of $\mathcal{M}$, then $f=\frac{ \|P_\HM k_\lambda\|}{\| k_\lambda\|}k_\lambda$ is the free outer factor of $g=P_\HM k_\lambda$. \end{example}
Indeed, one verifies that $P_f=P_g$. Hence since $f$ is a free outer function the statement follows from Theorem \ref{Pf=Pg}.

 By a Theorem of McCullough and Trent (\cite{McCulloughTrent}) $P_\HM$ can be associated with an inner sequence $\{\varphi_n\}$ so that $P_\HM = \sum_{n\ge 0} M_{\varphi_n}M_{\varphi_n}^*$, and hence $P_\HM k_\lambda= \sum_{n\ge 0} \overline{\varphi_n(\lambda)}\varphi_n k_\lambda$. This implies that the subinner factor of $P_\HM k_\lambda$ is
$$\frac{\sum_{n\ge 0} \overline{\varphi_n(\lambda)}\varphi_n }{\sqrt{\sum_{n\ge 0} |\varphi_n(\lambda)|^2}}.$$
The sequence $\varphi_n$ can also be used to give a representation of the subinner factor $\varphi$ of an arbitrary function $f\in \HM$. Indeed, if $f\in \HM$, then $f=P_{\HM}f=\sum_{n\ge 0} \varphi_n M_{\varphi_n}^* f$, where $\|f\|^2=\sum_{n\ge 0} \|M_{\varphi_n}^* f\|^2$. Then Theorem \ref{VectorFactorization} implies that there is a free outer function $g$ with $\|g\|^2= \sum_{n\ge 0} \|M_{\varphi_n}^* f\|^2$, and there are  $\psi_n \in \MuH$ such that $\sum_n \|\psi_n h\|^2\le \|h\|^2$ for all $h\in \HH$ and such that $M_{\varphi_n}^* f=\psi_n g$ for each $n\ge 0$. Then $f= \sum_{n\ge 0} \varphi_n \psi_n g$, and the uniqueness of the subinner/free outer factorization implies that the subinner factor of $f$ is $\varphi=\sum_{n\ge 0} \varphi_n \psi_n$.

We will now mention some applications of our results. A classical theorem of Caratheodory-Schur states that every function in the unit ball of $H^\infty(\D)$ is a pointwise limit of a sequence of finite Blaschke products, see e.g. \cite{Garnett}, Theorem 2.1. or \cite{Nikolski}, p. 36. In \cite{ShiASS} Shimorin  defined the class of subextremal functions for the Bergman space $L^2_a$, and he showed that every such function can be approximated by a pointwise limit of rational $L^2_a$-extremal functions. Since the subextremal functions contain the unit ball of $H^\infty(\D)$ we may consider this to be an analogue of the Caratheodory-Schur Theorem. On the other hand for the Dirichlet space it was shown in \cite{LuoRi} that there are contractive multipliers of the Dirichlet space that are not a pointwise limit of a sequence of $D$-extremal functions. The subinner functions turn out to be solutions to extremal Pick  problems (see Theorem \ref{ExtremalPick}), and hence they form a sufficiently rich class of functions to achieve such density.

\begin{theorem}\label{Caratheodory Approximation} If $\HH$ is separable, infinite dimensional, and has a normalized complete Pick kernel $k$,  then for every function $\varphi$ in the unit ball of $\MuH$, there is a sequence $\{\varphi_n\}$ of  subinner functions such that $\varphi_n(x)\to \varphi(x)$ for every $x\in X$.

If  $\HH\subseteq \Hol(\Bd)$ for $1\le d <\infty$ and the monomials $\{z^\alpha\}_{\alpha \in \N_0^d}$ form an orthogonal basis for $\HH$ and are multipliers of $\HH$, then  the approximating subinner functions can be chosen to be rational functions.
\end{theorem}

We will prove this in Section \ref{SecApproxBySubinner}. As a second application we state the following theorem.

\begin{theorem} \label{SpacesWithPickFactorIntro} Let $\HH_k$ and $\HH_s$ be separable Hilbert function spaces on a set $X$ such that $s$ is a  complete Pick kernel, which is normalized at $z_0\in X$, and  such that $k/s $ is positive definite.

  If $f\in \HH_k \setminus \{0\}$, then there is a unique pair of functions $\varphi\in \Mult(\HH_s,\HH_k)$ and $g\in \HH_s$ such that
\begin{enumerate}
\item $f=\varphi g$,
\item $\|f\|_{\HH_k}=\|g\|_{\HH_s}$,
\item $g$ is free outer with $g(z_0)>0$,
\item  $\|\varphi\|_{\Mult(\HH_s,\HH_k)}\le 1$.
\end{enumerate}
\end{theorem}

 The proof will use the vector version of Theorem \ref{subinner/free outer}, see Section \ref{SecVectorvalued}. The factors in the above factorization have some special properties as compared to general functions in $\HH_k$. In fact, functions in $\HH_s$ can be written as ratios of multipliers of $\HH_s$, and contractive multipliers from $\HH_s$ to $\HH_k$ e.g. satisfy the pointwise estimate $|\varphi(z)|\le \frac{\|k_z\|}{\|s_z\|}$, see \cite{AHMcCR_Factor}, Section 4.

 Theorem \ref{SpacesWithPickFactorIntro} applies for example to the situation where $\HH_k$ is the Hardy or Bergman space of $\Bd$ and $\HH_s=H^2_d$. In the case when $d=1$ one can give a direct proof of this result that is based on the sweep of a measure (as compared to the isometric dilation) and this argument extends to $P^t(\mu)$-spaces, $1\le t<\infty$, see Theorem \ref{Ptmu}.

Another application concerns the weak product of the Hilbert function space $\HH$. It is defined by
$$\HH\odot \HH= \Big\{\sum_{n\ge 1} f_ng_n: \sum_{n\ge 1}\|f_n\|\|g_n\|<\infty\Big\}.$$ $\HH \odot \HH$ becomes a Banach space with norm
$$\|h\|_{\HH\odot \HH}=\inf\Big\{\sum_{n\ge 1}\|f_n\|\|g_n\|: h=\sum_{n\ge 1} f_ng_n\Big\}.$$
It is known that $H^2(\partial \Bd)\odot H^2(\dB)=H^1(\dB)$ and there are similar results for Bergman spaces, see \cite{CRW}. We think of $\HH\odot \HH$ as an analogue of the space $H^1$ for the general Hilbert function space. For further motivation and information about weak products we refer the reader to \cite{ARSW_Bilinear} and \cite{RiSuWeakProd}. Also, see \cite{AHMcCR_WP} for the particular cases where $\HH$ has  a complete Pick kernel.

 A combination of  Theorem 1.3 of \cite{JuMa2} and of Theorem 1.2 of \cite{HartzColumnRow} implies that if a Hilbert function space has a normalized complete Pick kernel, then for every $h\in \HH\odot \HH$, there is a pair of functions $f,g\in \HH$ such that $h=fg$ and $\|h\|_{\HH\odot\HH}=\|f\|\|g\|$. It turns out that subinner/free outer pairs are relevant in this context as well.

\begin{theorem} \label{WeakProductNorm} Let $\HH$ be a Hilbert function space with normalized complete Pick kernel.
  If $h \in \mathcal{H} \odot \mathcal{H} \setminus \{0\}$, then there exists a   subinner/free outer pair $(\varphi,f)$ with
  $h = \varphi f^2$ and $\|h\|_{\mathcal{H} \odot \mathcal{H}}
  = \|f^2\|_{\mathcal{H} \odot \mathcal{H}} = \|f\|^2$.
\end{theorem}

\

Unfortunately, for general functions $f\in \HH$  we do not have a way to explicitly determine the free outer factor. In \cite{AHMcCR_Factor} the Sarason function \begin{align}\label{SarasonFunctionDef} V_f(z)= 2\la f, k_z f\ra -\|f\|^2\end{align} was used to establish an explicit formula to write $f$ as a ratio of two multipliers. Note that it was shown in \cite{GrRiSu} (also see Lemma \ref{Lemma7.2}) that for normalized complete Pick kernels for each $z\in X$ the function $k_z$ is a multiplier, thus the Sarason function is well-defined. If $\HH=H^2$ is the Hardy space of the unit disc, then $$V_f(z)= \int_{|w|=1} \frac{w+z}{w-z}|f(w)|^2 \frac{|dw|}{2\pi},$$ and hence $V_f$ determines the outer factor of $f$ uniquely.

For the larger generality of all spaces with complete Pick kernel we mentioned above that two functions $f,g \in \HH$ have the same free outer factor, whenever $f$ and $g$ define the same vector states $P_f=P_g$ (Theorem \ref{Pf=Pg}). Thus, any function $f$ has the same Sarason function as its free outer factor. Furthermore, in Section \ref{SecFreeSarason} we will define two natural extensions of $V_f$, the free Sarason functions $V_F^r$ and $V_F^\ell$. They are free functions defined for functions $F$ in the Fock space, and they uniquely determine the free outer factor of $f$, whenever $F$ is the image of $f$ under the Fock space embedding of $\HH$. Thus, one may wonder whether  $V_f$ itself determines the free outer factor of $f$. For many spaces this is the case, and in those cases the Sarason function provides a convenient way to keep track of when vector states agree.

In order to state our result we recall notation regarding multinomial notation. If $d\in \N$, then a multi-index $\alpha$ is an element in $\N_0^d,$  $\alpha=(\alpha_1, \dots, \alpha_d)$, and $|\alpha|=\sum_{j=1}^d \alpha_j$. If $d=\infty$, then a multi-index is a sequence of the form $\alpha=(\alpha_1,\alpha_2, \dots)$ with $\alpha_j\in \N_0$ for each $j \in \N$ and $|\alpha|=\sum_{j=1}^\infty \alpha_j <\infty$. We write $\underline{0}=(0,0,\dots)$ and $I_\infty$ for the collection of all such multi-indices. If $n\in \N$, then we will also write $$I_n=\{\alpha \in I_\infty: \alpha_j=0 \text{ for all }j>n\}.$$ Note that $I_\infty$ is a countably infinite set.

 If $d\in \N\cup\{\infty\}$ and if  $z\in \mathbb{B}_d$, then the expressions $z^\alpha= \prod_{j=1}^d z_j^{\alpha_j}$ and $\alpha!=\prod_{j=1}^d  \alpha_j!$ are well-defined for each $\alpha\in I_d$, because all products only have finitely many factors that are $\ne 1$.

 \begin{theorem}\label{SarasonFunction} Let $\HH$ be a Hilbert function space with normalized complete Pick kernel on a set $X$, and let $f,g \in \HH \setminus \{0\}$.

If either

(a) $f, g \in \MuH$ or

(b) $d\in \N\cup \{\infty\}$, $X=\Bd$, $k_w(z)=\sum_{\alpha\in I_d} a_\alpha z^\alpha \overline{w}^\alpha$, where  each $a_\alpha>0$ and $\sum_{\alpha\in I_d} a_\alpha |z^\alpha|^2<\infty$ for each $z\in \Bd$,

then  the following are equivalent:
\begin{enumerate}
\item $f$ and $g$ have the same free outer factor,
\item $P_f=P_g$,
\item  $V_f=V_g$.\end{enumerate}
\end{theorem}
In particular, the theorem applies to $H^2(\D)$, the Dirichlet space $D$, and the Drury-Arveson space $H^2_d$. Note further that if $X$ is finite, then every element of $\HH$ is a finite linear combination of reproducing kernels. Hence in that case all elements are multipliers and the theorem applies.

As far as we know, it is possible that the conclusion of the previous theorem is true for all complete Pick spaces. That would be the case, if the answer to the following question is positive.
\begin{ques} If $\HH$ is a Hilbert function space with normalized complete Pick kernel, then are  finite linear combinations of the reproducing kernels weak*-dense in $\MuH$?  \end{ques}
One further situation, where we know the conclusion of the previous theorem holds is as follows.

\begin{theorem} \label{H2 interpolating} If $X=\{\lambda_1, \lambda_2, \lambda_3, \dots\}\subseteq \D$ is a Blaschke sequence of distinct points  with $0\in X$, if $\HH=H^2(\D)|X$, then the linear combinations of reproducing kernels are weak*-dense in $\MuH$.
\end{theorem}
We will prove this in Section \ref{SecSarasonF}. By $H^2(\D)|X$ we mean the space $$\{f:X \to \C, \exists g\in H^2(\D) \ f=g|X\}$$ with norm $\|f\|=\inf\{\|g\|_{H^2}: g\in H^2(\D) \ f=g|X\}$. It is clear that this space is isomorphic to $\bigvee_{\lambda\in X} \{\frac{1}{1-\overline{\lambda}z}\}\subseteq H^2(\D)$.

\

We now give an overview of the set-up for the remainder of the paper. Section \ref{SecFock} contains a review of the needed background about free analytic functions, the Fock space $\Fock$, and its distinguished subspace $\mathcal{H}^2_d$, the symmetric Fock space. In Section \ref{SecFreeSarason} we define the free Sarason functions $V_F^\ell$ and $V_F^r$ and establish  their basic properties. For functions $F\in \mathcal{H}^2_d$ there is a close connection between $V_F$ as defined in (\ref{SarasonFunctionDef}) and the free Sarason functions (see Lemma \ref{SarasonFunctionsLemma}). Section \ref{SecFreeInnerOuter}  will start with a review of the free inner-outer factorization results of Arias-Popescu and Davidson-Pitts (\cite{AriasPopescu,DavidsonPitts}). Then, in Sections \ref{SecFreeInnerOuter} and \ref{*invariant} we will prove that the free Sarason functions determine the free left and right outer factors of $F\in \Fock$, and that they provide a way to  establish that for functions $F$ in $*$-invariant subspaces $\HH$ of $\mathcal{H}^2_d \subseteq \Fock$ their free left and right outer factors are determined by data that are available from within $\HH$ (see Theorem \ref{Sarason/PosFunct}). These results are put together in Section \ref{SecPickspaces} to prove Theorems \ref{subinner/free outer} and \ref{Pf=Pg}. In Section \ref{SecSarasonF} we will prove Theorems \ref{SarasonFunction} and \ref{H2 interpolating}, which  also imply  that for many complete Pick spaces the Sarason function $V_f$ uniquely determines the free Sarason functions $V_F^r$ and $V_F^\ell$.
In Sections \ref{SecApproxBySubinner} and \ref{SecPickProblem} we establish a generalized version of Theorem \ref{Caratheodory Approximation} and we prove that  solutions to extremal Pick problems are subinner. Sections \ref{SecDiriExample} and \ref{SectionDruryArveson} contain specific information and examples about the Dirichlet and Drury-Arveson spaces. For example, we will show that subinner functions $\varphi$ for the Dirichlet space have the property that $M_\varphi$ attains its norm only on a 1-dimensional subspace (Theorem \ref{DiriDim=1}). We will also see by example that products of free outer polynomials may or may not be free outer. Furthermore, we will show that free outer polynomials in $D$ necessarily have no zeros in a disc that properly contains the closed unit disc.

In Section \ref{SecVectorvalued} we will show that sequences of functions in $\HH$ whose norms are square summable have a unique common free outer factor. That result can be used to give quick proofs of Theorem \ref{SpacesWithPickFactorIntro}, of the result from \cite{DavidsonHamilton} that the multiplier algebra of a complete Pick space has property $\mathbb{A}_1(1)$, and that sums of Sarason functions are Sarason functions (Corollary \ref{A1(1)}). Property $\mathbb A_1(1)$ implies that if $\Phi$ is a weak*-continuous linear functional on $\MuH$ and if $\varepsilon >0$, then there are $f,g\in \HH$ such that \begin{align} \label{A1(eps)}\|f\|\|g\|\le \|\Phi\|+\varepsilon \ \text{ and } \Phi(u)=\la uf,g\ra \ \text{ for all }u \in \MuH.\end{align} In Theorem \ref{infimum} we will show that in (\ref{A1(eps)})  one can take $\varepsilon=0$, if and only if one can choose $f$ and $g$ of the form $g=\varphi f$ for some subinner/free outer pair $(\varphi,f)$.

Theorem \ref{WeakProductNorm} will be proved in Section \ref{SecWeakProd}. Finally, in Section \ref{Sec:SubinnerDirichlet} we assume $d<\infty$ and we  slightly generalize a result of Clou\^atre and Davidson (\cite{CD16}) and we show that for regular unitarily invariant complete Pick spaces on $\Bd$ a multiplier $\varphi$ of norm one is subinner, whenever $\|\varphi\|_\infty < \|\varphi\|_{\MuH}=1$ (Proposition \ref{prop:subinner_mult_norm}) and either it is a multiplier-norm limit of polynomials or the space satisfies the one-function Corona theorem and $M_\varphi$ is essentially normal.  It is known that the standard weighted Dirichlet spaces $\HD_\alpha$ satisfy the one-function Corona theorem, and we go on to provide a sufficient condition for $M_\varphi$ to be  essentially normal on $\HD_\alpha$, Theorem \ref{essentiallyNormal}. We conclude that in $\HD_\alpha, \alpha >0$ sufficient regularity of a non-constant function $\varphi$ implies $\|\varphi\|_\infty <\|\varphi\|_{\MuH}$ and that such $\varphi$ is subinner, if $\|\varphi\|_{\MuH}=1$, see Lemma \ref{lem:liminf} and Theorem \ref{DiriSubinner}. This theorem does not apply to $H^2_d$.

\textbf{Acknowledgement:} The authors are grateful to Georgios Tsikalas for sharing an observation related to
Section \ref{Sec:SubinnerDirichlet}.

\section{Free holomorphic functions and the free Fock space}\label{SecFock}

 In the following we summarize the  needed basic definitions and facts about free holomorphic functions and the free Fock space  $\Fock$. For proofs and further details we refer to \cite{GeluPop} \cite{GeluPopMultiA_I}, \cite{GeluPopMulti_II}, \cite{JuMa1}, and \cite{SaloShalitSham1}. Some care is required since different authors have used different definitions and some papers do not specifically include the case $d=\infty$. We mention that we will consider free holomorphic functions functions on the noncommutative ball (nc ball). That is the common approach now, but Popescu in the papers referenced above, considers the functions to be defined on the larger ``operatorial unit ball''.  For the results that are needed in this paper that never creates a problem, and the needed results follow by the same proofs as suggested in the referenced papers. Nevertheless, in the paragraphs below we occasionally offer some additional comments. Furthermore, we  refer to Section 3 of \cite{SaloShalitSham1} for proofs that many of  the relevant concepts can indeed be identified, and that includes the case when $d=\infty$.

Let $d\in \N$ or $d=\infty$, and let $F^+_d$ denote the free semi-group on $d$ generators. Thus, $F^+_d$ consists of the empty word $\emptyset$ and all words of finite length with letters from the alphabet $A_d=\{1, 2, \dots d\}$ with $d$ letters. If $w=w_1 w_2 \cdots w_n$ with $w_i\in A_d$, then the flip of $w$ is defined by $\tilde{w}=w_n \cdots w_2 w_1$.
The length of the empty word is defined to be 0, while if $w=w_1 w_2 \cdots w_n$ with $w_i\in A_d$, then we define the length of $w$ by $|w|=n$. If $w\in F^+_d$, then $\alpha(w)\in \N_0^d$ is the multi-index associated with $w$ and it is defined by $\alpha(w)=(\alpha_1,\dots, \alpha_d)$, where $\alpha_j$ equals the number of times the letter $j\in A_d$ occurs in $w$.

Now let $x=(x_1,...,x_d)$ be a freely non-commuting indeterminate with $d$ components. We will use it in formal computations with free polynomials and free power series. If $w\in F^+_d$,  then the free monomials are defined by $x^w=1$, if $w=\emptyset$, and $x^w=x_{w_1}\dots x_{w_k}$, if $w=w_1 \cdots w_k$. If $n\in \N$ and if $X=[X_1,\dots, X_d]$ is a $d$-tuple of $n \times n$ matrices, then we can evaluate the monomial $x^w$ at $X$ by forming the matrix $X^w= X_{w_1}\cdots X_{w_k}$.
A formal power power series in $x$ is an expression of the type $F(x)=\sum_{w\in F^+_d}a_w x^w$, where $a_w\in \C$ for each $w\in F^+_d$. Formal power series may converge, when evaluated at certain $d$-tuples of matrices. For $n\in \N$ and $0\le R \le \infty$ we write $\mathbb{B}_d^{n\times n}(R)$ for $d$-tuples $X=[X_1,\dots, X_d]$ of $n \times n$ matrices whose row operator norm $\|X\|$ as linear transformation $X\in \HB(\C^n\otimes \C^d, \C^n)$ (resp. $X\in \HB(\C^n\otimes \ell_2,\C^n)$ if $d=\infty$) is $<R$. Here $X: (u_1,\dots, u_d)^t \to \sum_{i=1}^d X_iu_i$ for $u_1,..., u_d \in \C^n$, and $\|X\|=\|[X_1,\dots, X_d]\|$ is defined by
$$\|X\|= \sup \Big\{\| \sum_{i=1}^d X_iu_i\|_{\C^n}: \sum_{i=1}^d\|u_i\|^2_{\C^n}\le 1 \Big\}.$$ The nc-ball of radius $R$ is defined to be
$$\mathbb{B}_d^{nc}(R)=\bigcup_{n=1}^\infty \mathbb{B}_d^{n\times n}(R)$$ and we with write $\ncB$ for the (open) nc-unit ball $\mathbb{B}_d^{nc}(1)$.

\begin{lemma}\label{radius of convergence} Let $d\in \N \cup\{\infty\}$, and let  $F(x)=\sum_{w\in F^+_d}a_w x^w$ be a free power series such that
\begin{equation}\label{SummabilityCondition}M_k(F):=\left(\sum_{|w|=k}|a_w|^2\right)^{1/2} <\infty \ \text{ for each }k\in \N \cup\{0\}.\end{equation}
Then for any $k \ge 0$, $n \in \N$ and any $d$-tuple  $X=[X_1,\dots, X_d]$ of $n\times n$ matrices that defines a bounded row operator the sum $\sum_{|w|=k} a_wX^w$ converges absolutely in $\HB(\C^n)$ and
\begin{equation}\label{coeffestimate} \Big\|\sum_{|w|=k}a_wX^w \Big\|\le M_k(F) \|X\|^k.\end{equation}

Define $R\in [0,\infty]$  by $$\frac{1}{R}= \limsup_{k\to \infty}\left( \sum_{|w|=k}|a_w|^2\right)^{\frac{1}{2k}}.$$ Then for $X\in \mathbb{B}_d^{n\times n}(R)$ we have $\sum_{k=0}^\infty \|\sum_{|w|=k} a_wX^w\|<\infty$ and hence $\sum_{k=0}^\infty \sum_{|w|=k} a_wX^w$ converges in $\HB(\C^n)$.
\end{lemma}
For $d<\infty$ this follows from Theorem 1.1 of \cite{PopescuJFA2006}.

\begin{proof} We may and will assume $R>0$. Let $n\in \N$ and $X =[X_1,\dots, X_d]\in \mathbb{B}_d^{n \times n}(R)$.
  We first show that for each $k\in \N\cup \{0\}$ and for $u\in \C^n$ \begin{equation}\label{sum}\sum_{|w|=k}|a_w|\|{X^w}^*u\|<\infty.\end{equation}
Note that $\|X\|=\|X^*\|$, hence for $u\in \C^n$ we have $\sum_{i=1}^d\|X_i^*u\|^2\le \|X\|^2 \|u\|^2$. It follows that an induction argument shows that for $k \ge 0$
$$\sum_{|w|=k}\|{X^w}^*u\|^2 \le \|X\|^{2k}\|u\|^2.$$ Thus, condition (\ref{sum}) follows easily from the summability assumption (\ref{SummabilityCondition}) on the coefficients of $F$, since
\begin{align*}\sum_{|w|=k}|a_w|\|{X^w}^*u\| &\le \left(\sum_{|w|=k}|a_w|^2 \right)^{1/2} \left(\sum_{|w|=k}\|{X^w}^*u\|^2\right)^{1/2}\\
&\le M_k(F) \|X\|^k\|u\|\\
\end{align*}
We note that this estimate also implies inequality (\ref{coeffestimate}).
Moreover, \eqref{sum} implies that the sum $\sum_{|w| = k} \overline{a_w} {X^w}^*$ converges
unconditionally in the strong operator topology.
Since $X$ consists of operators on a finite dimensional space, the sum $\sum_{|w| = k} a_w X^w$
converges unconditionally in the norm of $\mathcal{B}(\mathbb{C}^n)$.
Finite dimensionality of $\mathcal{B}(\mathbb{C}^n)$ then yields absolute convergence.

The remainder of this proof is routine. Let $r\in \R$ with $\|X\|<r<R$. Then there is $k_0\in \N$ such that whenever $k\ge k_0$, then $(M_k(F))^2=\sum_{|w|=k}|a_w|^2 \le r^{-2k}$. Thus,  by (\ref{coeffestimate}) we have
\begin{equation*}\sum_{k=0}^\infty \|\sum_{|w|=k}a_wX^w\| \le \sum_{k=0}^{k_0-1} M_k(F)\|X\|^k + \sum_{k=k_0}^\infty \frac{\|X\|^k}{r^k}<\infty. \qedhere
\end{equation*}
\end{proof}
The number $R$ is called the radius of convergence of the free power series.
If $R>0$, then it follows that for each $n \in \N$ and  $X\in \mathbb{B}_d^{n\times n}(R)$ the series $\sum_{k=0}^\infty \left(\sum_{|w|=k} a_w X^w\right)$ converges in $\HB(\C^n)$.

If $R>0$, then a function $F: \ncB(R)\to \C^{nc}:= \bigcup_{n=1}^\infty M_n(\C)$ is called a free holomorphic function if there are complex coefficients $\{a_w\}_{w\in F^+_d}$ such that the free power series $\sum_{w\in F^+_d} a_wx^w$ has radius of convergence $\ge R$ and $F(X)=\sum_{w\in F^+_d}a_wX^w$ for each $X \in \ncB(R)$. An important fact about free holomorphic functions $F$ on $\ncB(R)$ is that the coefficients $a_w$ are  determined uniquely by the values $F(X)$ for $X\in \ncB(R)$. This follows for example as in \cite{PopescuJFA2006}, page 277, with the help of the left creation operators $L=(L_1,\dots, L_d)$ on the free Fock space, by considering for $0<r<R$ the operators $X_n= P_n (rL)|\HH_n\in \ncB(R)$, where $\HH_n$ is the space of free polynomials of degree $\le n$ in the first $m=\min\{n,d\}$ variables, and $P_n$ is the orthogonal projection of the free Fock space onto $\HH_n$.

If $F(x)=\sum_{w\in F^+_d}\hat{F}(w)x^w$ and $G(x)=\sum_{w\in F^+_d}\hat{G}(w)x^w$ are free holomorphic functions in $\ncB(R)$ and if $a,b \in \C$, then $(aF+bG)(x)=\sum_{w\in F^+_d}(a\hat{F}(w)+b\hat{G}(w))x^w$ and $FG(x)=\sum_{w\in F^+_d}c_w x^w$, where $c_w= \sum_{u,v\in F^+_d, w=uv}\hat{F}(u)\hat{G}(v)$ are free holomorphic functions on $\ncB(R)$ with $(aF+bG)(X)=aF(X)+bG(X)$ and $(FG)(X)=F(X)G(X)$ for all $X\in \ncB(R)$.

The free Fock space $\Fock$ is the space of formal power series in $x=(x_1,\dots, x_d)$ with square summable coefficients i.e. $F\in \Fock$ if and only if $F(x)=\sum_{w\in F^+_d}\hat{F}(w) x^w$ and $\|F\|^2=\sum_{w\in F^+_d}|\hat{F}(w)|^2<\infty$. It follows from Lemma \ref{radius of convergence} that each $F\in \Fock$ has radius of convergence $\ge 1$, hence the free Fock space consists of free holomorphic functions on $\ncB$, and \eqref{coeffestimate} implies that
\begin{align}\label{FockPointwiseEstimate}\|F(X)\|^2\le \frac{\|F\|^2_{\HF^2_d}}{1-\|X\|^2} \text{ for all }F \in \HF^2_d, X \in \ncB.\end{align}
 A unitary operator $W$ on $\Fock$ which satisfies $W=W^{*}$ is defined via the flip operation: If $F= \sum_{w\in F^+_d}\hat{F}(w) x^w$, then $\tilde{F}=WF= \sum_{w\in F^+_d}\hat{F}(\tilde{w}) x^w$. We will call $W$ the flip operator.

 The non-commutative left analytic Toeplitz algebra is $${\MFockL}=\{G\in \Fock: \|G\|_\infty=\sup_{X\in \ncB}\|G(X)\|<\infty\}.$$ It turns out that $L_G: F\to GF$, multiplication by $G$ from the left, defines a bounded operator on $\Fock$, if and only if $G\in \MFockL$ and $\|G\|_\infty= \|L_G\|$;
 see Section 3 of \cite{SaloShalitSham1}.

 We will also use the notation $R_G$ for multiplication by $G$ from the right, and it is a simple calculation that shows that $R_{G}=WL_{\tilde{G}}W$, and hence $\|R_G\|=\|\tilde{G}\|_\infty$. Thus, it is natural to define the right Toeplitz algebra by $${\MFockR}=\{G\in \Fock: \tilde{G} \in \MFockL\}.$$

  Special cases of this are the left creation operators $L_i=L_{x_i}$ and the right creation operators $R_i=R_{x_i}$, which are defined by $(L_iF)(x)=x_iF(x)$ and $(R_iF)(x)=F(x)x_i$, $ i=1,\dots,d$. Finally, for $w\in F^+_d$ we will write $L^w=L_{x^w}$ and $R^w=R_{x^w}$.

The Toeplitz  algebras of operators $\{L_G\in \HB(\HF^2_d): G\in \HF_d^{\infty, \ell}\}$ and  $\{R_G\in \HB(\HF^2_d): G\in \HF_d^{\infty, r}\}$ are closed in both the weak* topology of  $\HB(\HF^2_d)$ and the weak operator topology, and the two topologies coincide on these Toeplitz algebras, see Corollary 2.12 of \cite{DavidsonPitts}. Furthermore, they have property $\mathbb{A}_1$, i.e. every weak*-continuous linear functional $L$ can be represented by a rank one operator $F\otimes G$ by $L(L_H)=\la L_H F,G\ra$ for all $H\in \MFockL$. Thus, these rank one operators induce  weak* topologies on the algebras $\MFockL$ and $\MFockR$.

\begin{lemma} \label{weakconvergence} (a) If $F_n, F\in \HF_d^2$, then $F_n$ converges weakly to $F$, if and only if $\|F_n\|\le C$ and $F_n(X) \to F(X)$ for all $X \in \ncB$.

(b) If $H_n, H\in \HF_d^{\infty, \ell}$, then $H_n$ converges to $H$ in the weak*-topology, if and only if $\|H_n\|_\infty \le C$ and $H_n(X) \to H(X)$ for all $X \in \ncB$.
\end{lemma}
Note that if $X\in \ncB$, then $F(X)\in \HB(\C^k)$ for some $k$, so the ``pointwise convergence'' in (a) and (b) can  be understood as either norm convergence or convergence in the weak operator topology.
\begin{proof}(a) If  $X=(X_1,\dots, X_d)\in \ncB$ is a tuple of $k\times k$-matrices, and if $u,v \in \C^k$, then note that as in the proof of Lemma \ref{radius of convergence} we have $$\sum_w |\la {X^w}^* v,u\ra|^2=\sum_{k=0}^\infty \sum_{|w|=k}|\la {X^w}^* v,u\ra|^2 \le \sum_{k=0}^\infty \|X\|^{2k}\|u\|^2\|v\|^2.$$ Thus, the free power series $K_{X,u,v}(x)=\sum_w \la {X^w}^* v,u\ra x^w$ defines an element in $\HF^2_d$, and one checks that
 $\la H,K_{X,u,v}\ra_{\HF^2_d}=\la H(X)u,v\ra$ for all $H\in \HF^2_d$.

 Now suppose $F_n\to F$ weakly in $\HF_d^2$. By the uniform boundedness principle,  $\|F_n\|\le C$  for some $C>0$ and all $n$, and by the remarks from the previous
 paragraph we conclude that for each $X \in \ncB$ we have $F_n(X) \to F(X)$ in the weak operator topology.

 In order to show that converse we note first that if $F\in \HF^2_d$ is orthogonal to all elements of the type $K_{X,u,v}$, then it must be the zero function and hence $F=0$. Hence finite linear combinations of elements  of the type $K_{X,u,v}$ are dense in $\HF^2_d$, and it is now a routine approximation argument to finish the proof.

 (b) Let $H_n, H\in \HF_d^{\infty, \ell}$ such that  $H_n$ converges to $H$ in the weak*-topology, then $\|L_{H_n}\|=\|H_n\|_\infty$ is bounded and $H_n = L_{H_n}1$ converges to $H=L_H 1$ weakly in $\HF^2_d$. Hence $H_n(X)\to H(X)$ for all $X \in \ncB$ by part (a).

 Conversely, if  $\|H_n\|_\infty \le C$ and $H_n(X) \to H(X)$ for all $X \in \ncB$, then for all $X\in \ncB$ and for all $F\in \HF^2_d$ we have $H_n(X)F(X)\to H(X)F(X)$. Thus, by (a) we conclude that $H_n F \to HF$ weakly in $\HF^2_d$ for all $F\in  \HF^2_d$. Hence $\la L_{H_n} F,G\ra \to \la L_{H} F,G\ra$ for all $F,G\in \HF^2_d$.
 Since $\|L_{H_n}\| \le C$ for all $n$, it follows that $H_n$ converges to $H$ in the weak*-topology.
\end{proof}

We say that $p$ is a free polynomial, if it is of the form $p(x)=\sum_w \hat{p}(w) x^w$, where only finitely many of the coefficients $\hat{p}(w)\ne 0$.

\begin{lemma}\label{freePolysWeak*dense} The free polynomials are weak*-dense in $\HF_d^{\infty, \ell}$ and in  $\HF_d^{\infty, r}$.\end{lemma}
\begin{proof} We will only consider the case $\HF_d^{\infty, \ell}$, the other case will follow after an application of the flip operator. Let $H \in \HF_d^{\infty, \ell}$, then it has a homogeneous expansion $H(x)= \sum_{k=0}^\infty H_k(x)$, where $H_k(x)= \sum_{|w|=k}\hat{H}(w)x^w$. On pages 405 and 406  of \cite{DavidsonPitts} Davidson and Pitts prove that the Fej\'{e}r means of the partial sums of the homogeneous expansion converge to $H$ in the weak*-topology. That proves the lemma in the case, when $d<\infty$. In order to also establish the lemma for $d=\infty$ it will suffice to show that functions that only depend on finitely many of the variables are weak*-dense in $\HF_\infty^{\infty, \ell}$.

Let $H\in \HF_\infty^{\infty, \ell}$, $H(x)=\sum_{w\in F^+_\infty} \hat{H}(w)x^w$. Then for $n\in \N$ define $H_n(x)= \sum_{w\in F^+_n}  \hat{H}(w)x^w$. Note that if $X=(X_1,X_2,\dots )\in \mathbb{B}^{nc}_\infty$, then so is $Y_n = (X_1,X_2,\dots, X_n, 0, \dots)$ and $\|X-Y_n\|^2=\|\sum_{k=n+1}^\infty X_k{X_k}^*\| \to 0$.

We have $H_n(X)=H(Y_n)$, hence $\|H_n\|_\infty \le \|H\|_\infty$.
Clearly, $H_n \to H$ in $\mathcal{F}^2_d$, so Lemma \ref{weakconvergence} shows that
$H_n(X) \to H(X)$.
\end{proof}

A distinguished subspace of $\Fock$ is the symmetric Fock space $\SymFock\subseteq \Fock$. An element $F(x)=\sum_w\hat{F}(w)x^w\in \Fock$ is in $\SymFock$, if and only if $\hat{F}(w)=\hat{F}(v)$, whenever $\alpha(w)=\alpha(v)$. For each multi-index $\alpha\in \N_0^d$ (which we understand as $I_\infty$ if $d = \infty$)
we have $\|\sum_{\alpha(w)=\alpha} x^w\|^2= \sum_{\alpha(w)=\alpha} \|x^w\|^2=\frac{|\alpha|!}{\alpha !}$, hence the set $\{e_\alpha\}_{\alpha \in \N_0^d}$, $e_\alpha= \sqrt{\frac{\alpha !}{|\alpha|!}} \sum_{\alpha(w)=\alpha} x^w$, forms an orthonormal basis for $\SymFock$. For $z\in \Bd$ set $$K_z=\sum_{w\in F^+_d} \overline{z}^{\alpha(w)}x^w=\sum_{\alpha\in \N_0^d}\overline{z}^\alpha\sum_{\alpha(w)=\alpha}x^w.$$ Then $K_z\in \SymFock$ and we can use $K_z$ to evaluate functions $F\in \Fock$ at $z\in \Bd$: $$F(z)=\la F, K_z\ra = \sum_w \hat{F}(w) z^{\alpha(w)}.$$ Since $F(z)$ is analytic in $z$ we  note that $F(z)=0$ for all $z\in \Bd$, if and only if $\sum_{\alpha(w)=\alpha}\hat{F}(w)=0$ for each $\alpha\in \N_0^d$; if $d = \infty$, then $\mathbb{B}_n \to \mathbb{C}, z \mapsto F(z)$ is analytic in for each $n \in \mathbb{N}$, which gives the same conclusion.
In turn, this happens if and only if
$F\in {\SymFock}^\perp$, since for all $\alpha\in \N_0^d$ we have $$\la F, e_\alpha\ra = \sqrt{\frac{\alpha !}{|\alpha|!}} \sum_{\alpha(w)=\alpha}\hat{F}(w).$$ This implies that the closed linear span of $K_z$, $z\in \Bd$ equals $\SymFock$.

 One calculates that $$\la K_z, K_\lambda \ra = \frac{1}{1-\la \lambda,z\ra}.$$ Now recall that $k_z(\lambda)= \frac{1}{1-\la \lambda,z\ra}$ is the reproducing kernel for the Drury-Arveson space $H^2_d$. It follows that the linear map defined by $Uk_z=K_z$ extends to be a unitary transformation of $H^2_d$ onto $\SymFock$, and if $f\in H^2_d$, then $(Uf)(z)=f(z)$ for each $z\in \Bd$.

 Furthermore, if $z\in \Bd$, $F\in \Fock$, and $G\in \MFockL$, then the identity $(L_GF)(z)=G(z)F(z)$ naturally leads to $L_G^*K_z=\overline{G(z)}K_z$. It follows that $\SymFock$ is invariant for each $L_G^*$ and $(L_G^*|\SymFock) U=UM_g^*$, where $g\in \Mult(H^2_d)$ is defined by $g(z)=G(z)$ for all $z\in \Bd$.

 Note that if $W$ is the flip operator, then $WK_z=K_z$ for all $z\in \B_d$ and thus
 $R_{\tilde{G}}^*K_z=WL_{G}^*WK_z=\overline{G(z)}K_z$. It follows that $R_{\tilde{G}}^*|\SymFock= L_G^*|\SymFock$ for each $G\in \MFockL$.

\section{The free Sarason functions}\label{SecFreeSarason}
\begin{definition} Let $F\in \Fock$.  We define the {\it left free Sarason function} of $F$ by the formal power series $$V^\ell_F(x)= 2\sum_{w\in F^+_d} \la F,R^w F\ra x^w -\|F\|^2 $$ and the {\it right free Sarason function} of $F$ by $$V^r_F(x)=2\sum_{w\in F^+_d} \la F,L^w F\ra x^w-\|F\|^2.$$
\end{definition}

\begin{lemma}\label{SarasonConv} Let $F, G\in \Fock$. Then the left and right free Sarason functions of $F$ and $G$ satisfy the summability hypothesis (\ref{SummabilityCondition}) and have radii of convergence $\ge 1$. Furthermore, $$\|V^r_F(X)-V^r_G(X)\|\le \sqrt{2}\|F-G\|\frac{(\|F\|^2+\|G\|^2)^{1/2}}{{1-\|X\|}} \text{ and }$$
$$\|V^\ell_F(X)-V^\ell_G(X)\|\le \sqrt{2}\|F-G\|\frac{(\|F\|^2+\|G\|^2)^{1/2}}{{1-\|X\|}} $$
 for all $X\in \ncB.$
\end{lemma}
\begin{proof} We  prove the statements for $V^r_F$, the others  follow analogously. We start by noting that $[L_1,L_2,\dots,L_d]$ defines a row isometry from $(\Fock)^d$ to $\Fock$. Hence its adjoint, the column operator $[L_1^*,\dots, L_d^*]^T$, is contractive, and thus an induction argument shows \begin{equation}\label{Sarason-kth-average}\sum_{|w|=k} \| {L^w}^*F\|^2\le \|F\|^2. \end{equation} Using the notation from Lemma \ref{radius of convergence} this implies
\begin{align*}M_k(V^r_F)^2&=\sum_{|w|=k} |\la F,L^w F\ra|^2\\&= \sum_{|w|=k} |\la {L^w}^*F, F\ra|^2\\
&\le \|F\|^2 \sum_{|w|=k} \| {L^w}^*F\|^2\\
&\le \|F\|^4.
\end{align*}
Thus $V^r_F$ satisfies the summability hypothesis (\ref{SummabilityCondition}) and by Lemma  \ref{radius of convergence} it has radius of convergence $\ge 1$. Furthermore,
\begin{align*} M_k(V^r_F-V^r_G)^2 &\le 2\sum_{|w|=k} |\la F-G,L^w F\ra|^2 + |\la G, L^w(F-G)\ra|^2\\
&\le 2\sum_{|w|=k} \| {L^w}^*(F-G)\|^2 \|F\|^2 + \|{L^w}^* G\|^2 \|F-G\|^2\\
&\le 2 \|F-G\|^2(\|F\|^2+\|G\|^2) \text{ by (\ref{Sarason-kth-average})}.
\end{align*}
By (\ref{coeffestimate}) of Lemma \ref{radius of convergence} this implies for $X \in \ncB$
\begin{align*}\|V^r_F(X)-V^r_G(X)\| &\le \sum_{k=0}^\infty M_k(V^r_F-V^r_G)\|X\|^k\\
&\le \sqrt{2} \|F-G\| \frac{\sqrt{\|F\|^2+\|G\|^2}}{1-\|X\|}. \qedhere\end{align*}
\end{proof}
The convergence of the series as proven in the Lemma implies that the coefficients $\{\la F,L^w F\ra \}_{w\in F^+_d}$ are uniquely determined by the free holomorphic function $V^r_F$ (analogously for $V^\ell_F$).
We note that if  $F\in \MFockL$, then since $\{x^w\}_{w\in F^+_d}$ is an orthonormal basis of $\Fock$ we see that $V^\ell_F(x)= 2 (L^*_F F)(x) -\|F\|^2$. Thus, the extra assumption that $F\in \MFockL$ implies that $V^\ell_F\in \Fock$. From the remarks at the end of Section \ref{SecFock}, one can deduce that if $F\in \SymFock$, then ${R^*}^wF={L^*}^wF$ implies that $V^r_F=V^\ell_F$. More is true in this case.

\begin{lemma} \label{SarasonFunctionsLemma} If $F\in \SymFock$ and $f\in H^2_d$ with $F=Uf$, then $V_F^r=V_F^\ell$ and $V^r_F(z)=V^\ell_F(z)=V_f(z)$ for all $z\in \Bd$, where $V_f$ is the Sarason function as defined in equation (\ref{SarasonFunctionDef}).

Furthermore, if $F,G\in \SymFock$, $f,g \in H^2_d$ with $F=Uf,G=Ug$, then $V^r_F=V^r_G$ as free holomorphic functions on $\ncB$ if and only if $V_f=V_g$ as analytic functions on $\Bd$\end{lemma}

\begin{proof} We start by noting that for all $w\in F^+_d$ we have $\la F, L^w F\ra = \sum_{u\in F^+_d}\hat{F}(wu)\overline{\hat{F}(u)}$. Thus, if $w_1, w_2 \in F^+_d$ with $\alpha(w_1)=\alpha(w_2)$, then for all $u\in F^+_d$ we have $\alpha(w_1u)=\alpha(w_2u)$, and so $F\in \SymFock$ implies that \begin{align}\label{equation}\la F, L^{w_1}F\ra = \la F, L^{w_2}F\ra, \text{ whenever }\alpha(w_1)=\alpha(w_2).\end{align}

In order to show that $V_F^r=V_F^\ell$ we use the flip operator $W$. Since $\alpha(w)=\alpha(\tilde{w})$ for all $w\in F^+_d$ we have $F=\tilde{F}$ for all $F\in \SymFock$ and
\begin{align*}V_F^r(x) &=2\sum_{w\in F^+_d} \la F, L^{\tilde{w}} F\ra x^w -\|F\|^2 \ \ \text{ by (\ref{equation})} \text{ and since } \alpha(\tilde{w})=\alpha(w)\\
&= 2\sum_{w\in F^+_d}  \la \tilde{F}, R^w \tilde{F}\ra x^w -\|F\|^2 \ \ \text{ since } L^{\tilde{w}} =WR^w W\\
&= 2 \sum_{w\in F^+_d}  \la F, R^w F \ra x^w-\|F\|^2\\
&= V_F^\ell(x)\end{align*}

The equality (\ref{equation}) also means that $V^r_F$ is determined by its values at the points $z\in \Bd$, and hence for $F,G \in \SymFock$ we have $V^r_F=V^r_G$ as free holomorphic functions, if and only if $V^r_F(z)=V^r_G(z)$ for all $z\in \Bd$. Thus, it remains to be shown that $V^r_F(z)=V_f(z)$, where $f\in H^2_d$ satisfies $F=Uf$ and hence $F(\lambda)=f(\lambda)$ for all $\lambda\in \Bd$.

First suppose that $F\in \MFockL \cap \SymFock$. Then
\begin{align*}  V^r_F(z)&= 2 \la L_F^* F, K_z\ra -\|F\|^2\\
&=2 \la M_f^* f, k_z\ra_{H^2_d}-\|f\|^2_{H^2_d}\\
&= 2\la f, k_z f\ra_{H^2_d} -\|f\|^2_{H^2_d}\\
&=V_f(z)
\end{align*}
Thus the required identity $V^r_F(z)=V_f(z)$ holds on a dense subset of $\SymFock$, hence by Lemma \ref{SarasonConv} and the more elementary fact that $p_n\to f$ in $H^2_d$ implies $V_{p_n}(z) \to V_f(z)$ pointwise, it holds for all $F\in \SymFock$.
\end{proof}

\section{A closer look at the free inner-outer factorization}\label{SecFreeInnerOuter}

We start this section by stating the theorem about the free inner-outer factorization of elements of $\Fock$, see \cite{AriasPopescu}, Theorem 2.1, or \cite{DavidsonPitts}, Corollary 2.3.
These results were stated for $d < \infty$, but they can be extended to cover the case $d= \infty$. In fact, in \cite{GeluPop}, Theorem 4.2, Popescu had already established a closely related theorem, which included the possibility that $d=\infty$.

An element $G\in \MFockR$ is called right inner, if $R_G$ is isometric, and $F\in \Fock$ is called left outer, if $\{L_G F: G\in \MFockL\}$ is dense in $\Fock$. By Lemma \ref{freePolysWeak*dense} it follows that $F$ is left outer, if and only if $\{pF: p \text{ is a free polynomial}\}$ is dense in $\Fock$.

 Similarly,  $G\in \MFockL$ is left inner, whenever $L_G$ is an isometry, and $F\in \Fock$ is right outer, if $\{R_G F: G\in \MFockR\}$ or equivalently $\{Fp: p \text{ is a free polynomial}\}$ is dense in $\Fock$.
 \begin{theorem} \label{inner-outer} Every $H\in \Fock \setminus \{0\}$ has a left inner/right outer and a right inner/left outer factorization. The factorizations are unique up to a multiplicative scalar of modulus 1.
\end{theorem}
In particular we note that $F\in \Fock$ is right outer, if and only if it satisfies the following condition: whenever $F\in \ran L_\Phi$ for some left inner multiplier $\Phi\in \MFockL$, then $\Phi=c$ for some $c \in \mathbb{C}$ with $|c| = 1$.

Recall that a free polynomial is of the form $p(x)=\sum_w \hat{p}(w) x^w$, where only finitely many of the coefficients $\hat{p}(w)\ne 0$. If $p$ is a free polynomial, then it follows easily that the free Sarason functions are free polynomials as well, and hence they will define  bounded left and right multipliers.

\begin{lemma} \label{Lemma 0} If $F\in \Fock$, then \[\|pF\|^2= \operatorname{Re} \la F, V^\ell_p F\ra\] for every free polynomial $p$.
\end{lemma}
\begin{proof} If $p=\sum_{w}\hat{p}(w)x^w$, then $\la p,R^vp\ra = \sum_{w\in F^+_d} \hat{p}(wv)\overline{\hat{p}(w)}$, and hence \begin{equation}\label{1} \la F, V^\ell_p F\ra = 2\sum_{v,w\in F^+_d} \overline{\hat{p}(wv)}{\hat{p}(w)}\la F, L^vF\ra -\|p\|^2\|F\|^2.\end{equation}

 Next note that $\la L^w F, L^{w'} F\ra=0$ unless $w=w'v$ or $w'=wv$ for some $v \in F^+_d$. Thus, using the fact that each $L_i$ is isometric we calculate
\begin{align*} \|pF\|^2&= \sum_{w,w'\in F^+_d}\hat{p}(w)\overline{\hat{p}(w')}\la L^w F,L^{w'} F\ra\\
&=2 \text{ Re }\sum_{w, v\in F^+_d} \hat{p}(w)\overline{\hat{p}(wv)}\la  F,L^{v} F\ra- \|p\|^2\|F\|^2\\
&=\text{Re } \la F, V^\ell_p F\ra \ \text{  by identity (\ref{1})}. \qedhere
\end{align*}
\end{proof}
\begin{lemma}\label{Lemma 1} If $F, G\in \Fock$, then $V_F^r=V_G^r$ if and only if
 $$\|pF\|^2= \|pG\|^2$$ for every free polynomial $p$.
\end{lemma}
\begin{proof} If $V_F^r=V_G^r$, then the functions have the same power series coefficients and hence $\la F,q F\ra=\la G, qG\ra$ for each free polynomial $q$. The conclusion follows immediately from the previous lemma since $V_p^\ell$ is a free polynomial for each free polynomial $p$, so
$$\|pF\|^2 = \text{ Re } \la F, V_p^\ell F\ra=\text{ Re }\la G, V_p^\ell G\ra = \|pG\|^2.$$

In order to see that the converse is true as well, we use polarization to see that  $\la F,q F\ra=\la G, qG\ra$ for each free polynomial $q$. This implies that $V^r_F=V^r_G$.
\end{proof}
Lemma \ref{Lemma 1} says that $\|p F\|^2 = \|p G\|^2$ for every free polynomial $p$ if and only if
$\langle F, L^w F \rangle = \langle G, L^w G \rangle$ for all $w \in F_d^+$.
There are other, perhaps more elementary calculations that lead to this conclusion, but we wanted to show the connection with the Sarason functions. In fact,   Lemma \ref{Lemma 0} together with an elementary calculation implies that $$\|(V^\ell_p-1)F\|^2+4\|pF\|^2=\|(V^\ell_p+1)F\|^2,$$ and a related expression was the basis of the proof of Proposition 3.5 of \cite{AHMcCR_Factor}.

\begin{theorem}\label{inner} Let $\Phi\in \MFockR$ with $\|R_\Phi\| \le 1$. If  there is a $F\in \Fock$, $F \ne 0$, such that $\| F \Phi\|=\|F\|$, then $R_\Phi$ is isometric when restricted to $[F]_\ell= \clos \{L_\Psi F: \Psi \in \MFockL\}$.
\end{theorem}
In particular, it follows that if $F$ is left outer, then $\Phi$ is right inner.
\begin{proof} Assume that $F\in \Fock$, $F \ne 0$, $\Phi \in \MFockR$ with $\|R_\phi\| \le 1$ and   such that $\|F\Phi\|=\|F\|$. Then $R_\Phi^*  R_\Phi F= F$ by Lemma \ref{contractionLemma}, hence for all $w\in F^+_d$ we have $$\la F, L^w F\ra= \la R_\Phi^*  ( F \Phi), L^wF\ra = \la F\Phi , R_\Phi L^w F\ra =\la F \Phi , L^w(F \Phi)\ra.$$

  Hence $V^r_F=V^r_{F \Phi}$, and by Lemma \ref{Lemma 0} this implies that $\|pF\|=\|pF\Phi\|=\|R_\Phi(pF)\|$ for every free polynomial $p$. An application of Lemma \ref{freePolysWeak*dense} now finishes the proof.
\end{proof}
\begin{lemma} \label{Lemma 2} Let $F, G\in \Fock$ with $V^r_F=V^r_G$.

 (a) If  $F$ is left outer, then there is a right inner  $\Phi\in \MFockR$ such that  $G=R_\Phi F=F\Phi$.

(b) If both $F$ and $G$ are left outer, then
$F=cG$ for some $c\in \C$ with $|c|=1$.
\end{lemma}

Of course, an analogous lemma holds for right outer elements.
\begin{proof} (a)
    Since $F$ is left outer, $\{p F: p \text{ free polynomial}\}$ is dense in $\mathcal{F}^2_d$,
    so by Lemma \ref{Lemma 1}, there exists
    an isometry $V$ on $\mathcal{F}^2_d$ satisfying
    \begin{equation*}
      V (p F) = p G
    \end{equation*}
    for all free polynomials $p$. If $p, q$ are free polynomials, then $qV(pF)=qpG=V(qpF)$ and this implies that $qV1=Vq$ for every free polynomial $q$.

    Set $\Phi=V1$ and choose a sequence of free polynomials $q_n$ that converge to $F$ in $\Fock$. Then $q_n\Phi=Vq_n \to VF=G$ in $\Fock$. Furthermore the sequence converges pointwise in $\ncB$. We thus  conclude that $G=F\Phi$. Multiplying that identity by a free outer polynomial $p$ on the left we obtain
    $$V(pF)=pG=pF\Phi.$$ Hence multiplication by $\Phi$ on the right defines an isometric operator, i.e. $\Phi$ is right inner.

(b) If we also assume that $G$ is left outer, then the range of $R_\Phi $ will be dense in $\Fock$ and hence $R_\Phi $ is unitary. By Corollary 1.5 of \cite{DavidsonPitts}, or by the uniqueness part of Theorem \ref{inner-outer}, this implies that $R_\Phi $ is a constant.
\end{proof}
\begin{theorem} \label{Sarason/PosFunct} Let $F, G\in \Fock \setminus \{0\}$, then  the following are equivalent
\begin{enumerate}
\item   $V^r_F=V^r_G$,
\item $\la L_\Phi F,F\ra=\la L_\Phi G, G\ra$ for all  $\Phi\in \MFockL$,
\item $F=F_o\Phi_1$ and $G=F_o\Phi_2$, where  ${\Phi_1}$ and ${\Phi_2}$ are  right inner multipliers and $F_o$ is left outer.
\end{enumerate}
 \end{theorem}
 \begin{proof} The implications (iii) $\Rightarrow$ (ii) $\Rightarrow$ (i) are trivial, they follow immediately from the definitions. In order to show (i) $\Rightarrow$ (iii)  we note that by the existence of the inner-outer factorization (Theorem \ref{inner-outer}) we have $F=F_o\Phi_1$ and $G=G_o\Phi_2$, where $R_{\Phi_1}$ and $R_{\Phi_2}$ are isometric right multipliers and $F_o$, $G_o$ are left outer. If condition (i) is satisfied, then
 $$V^r_{F_o}=V^r_F=V^r_G=V^r_{G_o},$$
 and hence part (b) of the previous Lemma implies (iii).
 \end{proof}
 Of course, by symmetry the analogous Theorem holds with right and left multiplications switched. We state the result for later reference.
 \begin{theorem} \label{Sarason/PosFunct/LeftRight} Let $F, G\in \Fock \setminus \{0\}$, then  the following are equivalent
\begin{enumerate}
\item   $V^\ell_F=V^\ell_G$,
\item $\la R_\Phi F,F\ra=\la R_\Phi G, G\ra$ for all  $\Phi\in \MFockR$,
\item $F=\Phi_1F_o$ and $G=\Phi_2F_o$, where  ${\Phi_1}$ and ${\Phi_2}$ are  left inner multipliers and $F_o$ is right outer.
\end{enumerate}
 \end{theorem}

 \section{\texorpdfstring{$*$-invariant subspaces of $\SymFock$}{Star invariant subspaces of F2d}}\label{*invariant}

We say that a subspace $\HH \subseteq \SymFock$ is $*$-invariant, if $L_\Phi^* \HH \subseteq \HH$ for all $\Phi \in \MFockL$.
\begin{lemma} \label{SymFock} Let $H\in \SymFock \setminus \{0\}$, and suppose $H=\Phi_{\ell}F_r=F_{\ell}\Phi_r$ are its left inner/right outer and right inner/left outer factorizations, normalized
  so that $F_r(0) > 0$ and $F_\ell(0) > 0$. Then

(a) $F_r=F_{\ell}\in \SymFock$ and

(b) $P_\SymFock L_{\Phi_\ell}|\SymFock= P_\SymFock R_{\Phi_r}|\SymFock \cong M_g$, where $g\in \Mult(H^2_d)$ satisfies $g(z)=\Phi_r(z)=\Phi_{\ell}(z)$ for all $z\in \Bd$.

It follows that $H$ is right outer, if and only if it is left outer.
Furthermore, if $\HH \subseteq \SymFock \subseteq \Fock$ is a $*$-invariant subspace such that $H\in \HH$, then $F_r=F_{\ell}\in \HH$.
\end{lemma}
\begin{proof} First we note that for any $*$-invariant subspace $\HH$ with $H\in \HH$ we have $F_r=L^*_{\Phi_\ell}H\in \HH$ and $F_\ell=R^*_{\Phi_r}H \in \HH$. In particular, $F_r, F_\ell \in \SymFock$. By the implications (iii) $\Rightarrow$ (i) of Theorems \ref{Sarason/PosFunct} and \ref{Sarason/PosFunct/LeftRight} applied with $H$ and $F_\ell$ (resp. $H$ and $F_r$) we have $V^r_H=V^r_{F_\ell}$ and $V^{\ell}_H=V^\ell_{F_r}$. As noted in Lemma \ref{SarasonFunctionsLemma}, for functions in $\SymFock$ the left and right Sarason functions agree, hence $V^r_{F_r}=V^r_{F_\ell}$. Thus, by  Lemma \ref{Lemma 2}(a)  we conclude that $F_r= F_\ell \Phi_1$  for a right inner multiplier $\Phi_1 \in \MFockR$. Analogously, $F_\ell=\Phi_2 F_r$ for a left inner multiplier $\Phi_2\in \MFockL$. Then $F_r=\Phi_2 F_r \Phi_1$. Thus the right outer function $F_r$ is in the range of the left inner multiplier $L_{\Phi_2}$, hence $\Phi_2=1$ by the remark in the paragraph following Theorem \ref{inner-outer} and the normalization hypothesis. This proves (a).

Also, it implies that for each $z\in \Bd$ we have $H(z)=F_\ell(z) \Phi_r(z)= \Phi_\ell(z)F_\ell(z)$. Since $F_\ell(z)\ne 0$ for all $z\in \Bd$ we conclude that $\Phi_r(z)= \Phi_\ell(z)$ and hence $L^*_{\Phi_\ell}|\SymFock=R^*_{\Phi_r}|\SymFock$. Thus  (b) follows as well.

Now if $H$ is right outer, then $\Phi_\ell=c$ where $|c|=1$. Then $H(z)=cF_r(z)=F_\ell(z)\Phi_r(z)$ and the identity $F_r=F_\ell$ implies $\Phi_r=c$. Hence $H=F_\ell$ is left outer. The converse direction follows by symmetry.
\end{proof}

Now we extend the definition of free outer  functions to elements in $*$-invariant subspaces of $\SymFock$.
If $F \in \mathcal{H}^2_d$, we let $P_F$ denote the linear functional on $\MFockL$ defined by
\begin{equation*}
  P_F (\Phi) = \langle L_\Phi F, F \rangle.
\end{equation*}

\begin{definition}\label{free-outer-in*invariantDef} Let $\HH\subseteq \SymFock$ be a $*$-invariant subspace.
  If $F\in \HH \setminus \{0\}$, then $F$ is called $\HH$-free outer, if $$|F(0)|=\sup\{|G(0)|: G\in \HH, P_F=P_G\}.$$
\end{definition}

\begin{theorem} \label{free-outer-in*invariant} Let $\HH\subseteq \SymFock$ be a $*$-invariant subspace, and let $F\in \HH$. Then the following are equivalent:
\begin{enumerate}
\item $F$ is $\HH$-free outer,
\item $F$ is right outer in $\Fock$,
\item $F$ is left outer in $\Fock$.
\end{enumerate}
\end{theorem}
\begin{proof} The equivalence of (ii) and (iii) is part of Lemma \ref{SymFock}. Suppose that $F$ is left outer and that $G\in \HH$ is any function satisfying $P_G=P_F$. Then by (a trivial implication in) Theorem \ref{Sarason/PosFunct} we have $V^r_G=V^r_F$ and hence Lemma \ref{Lemma 2}(a) implies that $G= F\Phi$ for some right inner multiplier $\Phi\in \MFockR$. Then $|G(0)|=|F(0)\Phi(0)|\le |F(0)|$ and hence $F$ is $\HH$-free outer.

Conversely, suppose that $F$ is $\HH$-free outer, and consider its left inner/right outer factorization $F=\Phi G$. Then $P_F=P_G$ and $G\in \HH$ by Lemma \ref{SymFock}. Hence the fact that $F$ is $\HH$-free outer implies that $|\Phi(0)|=1$. Hence $1=|\Phi(0)|\le \|\Phi\|_{\HF_d^2}\le \|\Phi\|_\infty=1$, and this implies that $\Phi$ is constant and hence $F$ is right outer. \end{proof}

With the following Theorem we will link certain inner multipliers of $\Fock$ to subinner functions on spaces with complete Pick kernel. If $\HH \subseteq \SymFock$ is a $*$-invariant subspace, then we will write $P_\HH$ for the projection of $\Fock$ onto $\HH$ and we note that the fact that $L_\Phi^*|\SymFock=R_{\tilde{\Phi}}^*|\SymFock$ implies that $P_\HH \Phi F= P_\HH F \tilde{\Phi}$ for all $F\in \HH$ and $\Phi\in \MFockL$.
\begin{theorem}\label{inner-in*invariant}  Let $\HH\subseteq \SymFock$ be a $*$-invariant subspace, and let $\Phi\in \MFockL$, $\|\Phi\|_\infty \le 1$. Then the following are equivalent:

(a) $\Phi$ is left inner and $\ran L_\Phi \cap \HH \ne (0)$,

(b) $\tilde{\Phi}$ is right inner and $\ran R_{\tilde{\Phi}} \cap \HH \ne (0)$,

(c) there is $F\in \HH$ such that $\|P_\HH (\Phi F)\|=\|F\|\ne 0$,

(d) there is an $\HH$-free outer function $F\in \HH$ such that $\|P_\HH (\Phi F)\|=\|F\|\ne 0$.
\end{theorem}

 \begin{proof} Let $W$ denote the flip operator, then the identity $WL_\Phi W= R_{\tilde \Phi}$ implies that $\Phi$ is left inner, if and only if $\tilde{\Phi}$ is right inner.
   Furthermore, if $G=F \tilde{\Phi} \in \ran R_{\tilde{\Phi}} \cap \HH$, then $G=\tilde{G}$ since $\HH\subseteq \SymFock$, and hence $G=\tilde{G} = \Phi \tilde{F} \in \ran L_{{\Phi}} \cap \HH$. Thus, (b) $\Rightarrow $(a), and the converse holds similarly.

 We will now show the implications (a) $\Rightarrow$ (c) $\Rightarrow$ (d) $\Rightarrow$ (b). Suppose (a) holds. Then there is $F\in \Fock$, $F \ne 0$, such that $\Phi F=H\in \HH$. Then $\Phi F= P_\HH (\Phi F)$ implies $\|P_\HH (\Phi F)\|=\|\Phi F\|=\|F\|$ and the $*$-invariance of $\HH$ implies that $F=L^*_\Phi H\in \HH$. Thus (c) holds.

 Next we assume that (c) holds, i.e. there is a nonzero $F\in \HH$ such that $\|P_\HH (\Phi F)\|=\|F\|$. Then $$\|F\|=\|P_\HH (\Phi F)\|\le\| \Phi F\|\le \|F\|,$$ which means that $\|P_\HH (\Phi F)\|=\| \Phi F\|$ and hence $\Phi F \in \HH$.  Let $F=\Psi G$ be the left inner/right outer factorization of $F$. Then $G= L^*_\Psi F\in \HH$ as well, $G \ne 0$ is $\HH$-free outer by Theorem \ref{free-outer-in*invariant}, and
 \begin{align*}\|G\|&=\|F\|=\|P_\HH(\Phi \Psi G)\|\\
   &= \|P_\HH (\Psi G \tilde \Phi)\| \text{ by the remark before the statement} \\
   &= \|P_\HH L_\Psi P_\HH (G \tilde \Phi)\| \text{ since }\HH \text{ is co-invariant for left multipliers }\\
 &= \|P_\HH L_\Psi P_\HH (\Phi G)\| \\
 &\le \|P_\HH (\Phi G)\| \\
 &\le \|G\|
 \end{align*}
 Thus (d) holds with the function $G$.

 Now assume that (d) holds. Then as remarked before the statement of the Theorem we have $\|P_\HH(F\tilde{\Phi})\|=\|P_\HH(\Phi F)\|=\|F\|$, and hence with an argument as in  (c)$\Rightarrow$(d) we see that $ F {\tilde{\Phi}}\in \HH$ and  $\| F\tilde{\Phi}\|=\|F\|$. By Theorem \ref{free-outer-in*invariant} $F$ is left-outer, hence by Theorem \ref{inner} $\tilde{\Phi}$ is right inner with $F\tilde{ \Phi} \in \ran R_{\tilde{\Phi}} \cap \HH$, i.e.\ (b) holds.
 \end{proof}
\section{Normalized complete Pick spaces}\label{SecPickspaces}

  In this Section we will prove the main Theorems that were stated in Section 1. If $k_x(y)=\frac{1}{1-\la u(y),u(x)\ra}$ for some function $u:X \to \Bd$ for $d\in \N\cup\{\infty\}$ with $u(x_0)=0$, then by a theorem of Agler-McCarthy \cite{AgMcC_completePickKernels} we can identify $\HH_k$ with the $*$-invariant subspace $$\HH=\text{ closed linear span of }\{K_z : z\in \ran u\}\subseteq \SymFock\subseteq \Fock.$$ As in the earlier sections for $z\in \Bd$ we have used $K_z$ to denote the point evaluation functional, $K_z\in \SymFock \subseteq \Fock$, which for $z, w \in \Bd$ satisfies $K_z(w)=\frac{1}{1-\la w,z\ra}$.  The identity $k_x(y)=K_{u(x)}(u(y))$ can be used to show that the map $Uk_x =K_{u(x)}$ extends to be a linear isometry $U:\HH_k\to \Fock$ with range equal to $\HH$. Thus, $UU^*=P_\HH$, the projection onto $\HH$, and we have $U^*=C_u$,  where $C_uF(x)=F(u(x))$ for all $x\in X$.

 Furthermore, if $\Phi \in \MFockL$, then the definition $\varphi(x)=\Phi(u(x))$ defines a $\varphi \in \Mult(\HH_k)$ with $UM^*_\varphi=L^*_\Phi U$, and hence $M_\varphi$ is unitarily equivalent to $P_\HH L_\Phi|\HH$, and $\|\varphi\|_{\Mult(\HH_k)}\le \|\Phi\|_\infty$. Conversely, if $\varphi\in \Mult(\HH_k)$, then there is $\Phi \in \MFockL$ such that $\|\Phi\|_\infty=\|\varphi\|_{\Mult(\HH_k)}$  and $\varphi(x)=\Phi(u(x))$. We call any such $\Phi$ a lift of $\varphi$, see \cite{Frazho}, \cite{PopescuCommutantL}. Another relevant reference in this context is \cite{DavidsonLe}.

 Note that for a given complete Pick kernel $k$, the function $u$ is not unique. Indeed, if $V$ is any unitary operator on $\C^d$ (or $\ell^2$ if $d=\infty$), then the function $v: X \to \Bd$, $v(x)=Vu(x)$ can be used to define the same Pick kernel $k$. The different $v$ will lead to a different embedding of $\HH_k$ in $\Fock$. Throughout this section we will keep the $u$ fixed. Any uniqueness statement about liftings are to be understood as uniqueness statements for this fixed embedding.

\begin{lemma}
  \label{lem:free_outer_H}
  In the setting above, let $h \in \mathcal{H}_k \setminus \{0\}$ and let $H = U h \in \mathcal{H}^2_d$.
  Then $h$ is free outer in $\mathcal{H}_k$ if and only if $H$ is $\mathcal{H}$-free outer.
\end{lemma}

\begin{proof}
  Let $h_1,h_2 \in \mathcal{H}_k$ and define $H_i = U h_i$ for $i=1,2$.
  As explained above, $\varphi \in \Mult(\mathcal{H}_k)$ if and only if $\varphi = \Phi \circ u$ for some $\Phi \in \MFockL$.
  In this case,
  \begin{equation*}
    \langle \Phi H_i, H_i \rangle = \langle P_{\mathcal{H}} L_\Phi H_i , H_i \rangle = \langle \varphi h_i , h_i \rangle.
  \end{equation*}
  Hence $P_{h_1} = P_{h_2}$ if and only if $P_{H_1} = P_{H_2}$.
  Moreover, $H_i(0) = h_i(x_0)$. The result follows from these two observations.
\end{proof}

\begin{proof}[Proof of Theorem \ref{subinner/free outer}] The existence is just Theorem \ref{inner-outer} applied to functions in $\HH$ combined with the results of the previous section. Indeed, if $f\in \HH_k$, then $f=F\circ u$ for  $F=Uf\in \HH$. Then $F=\Phi G$ for some left inner multiplier $\Phi$ and a right outer function $G$ (Theorem \ref{inner-outer}). Theorem \ref{free-outer-in*invariant} shows that $G=L_{\Phi}^*F\in \HH$ is $\HH$-free outer.
  Let $g = U^* G = G \circ u$. Then $G = U g$ and $g$ is free outer in $\mathcal{H}_k$ by Lemma \ref{lem:free_outer_H}.
We also observe that $\|g\|=\|G\|=\|\Phi G\|=\|F\|=\|f\|$, hence if we set  $\varphi= \Phi\circ u$, then $f=\varphi g $ and $\varphi$ is subinner by definition.

Now suppose $f=\varphi g\in \HH_k$, where $(\varphi, g)$ is any subinner/free outer pair. We set $F=Uf, G=Ug$ and we use the lifting theorem to deduce the existence of a $\Phi \in \MFockL$ with $\|\Phi\|_\infty=1$ and $\Phi\circ u= \varphi$. Then  $F= P_\HH \Phi G\in \HH$ with $\|P_\HH \Phi G\|=\|F\|=\|G\|$. By Theorem \ref{inner-in*invariant} $\Phi$ is  left inner and since $\|G\|=\|P_\HH \Phi G\|\le \|\Phi G\|\le \|G\|$
we must have $\Phi G=P_\HH \Phi G$. Thus, $F=\Phi G$ is a left inner/right outer factorization of $F$, which we know to be unique up to a multiplicative unimodular constant (see Theorem \ref{inner-outer}). That implies that the factorization $f=\varphi g\in \HH_k$ was unique since $g(x_0)>0$, see Definition \ref{subinnerDef}.
 \end{proof}
\begin{proof}[Proof of Theorem \ref{Pf=Pg}] Let $f,g\in \HH_k$ have the same free outer factor $h$, then $f=\varphi h$, $g=\psi h$ for subinner functions $\varphi, \psi$ and $\|f\|=\|h\|=\|g\|$. Then the identities $h= M^*_\varphi f = M^*_\psi g$ imply that $P_f=P_h=P_g$, see Lemma \ref{contractionLemma}.

Conversely, we suppose that $f,g \in \HH_k$ with $P_f=P_g$.
As in the proof of Lemma \ref{lem:free_outer_H}, we find that $\langle \Phi F, F \rangle = \langle \Phi G ,G \rangle$ for all $\Phi \in \MFockL$.
Hence by Theorem \ref{Sarason/PosFunct} there is a left outer factor $H$ and right inner factors $\Psi_1$ and $\Psi_2$ such that $F=H\Psi_1$ and $G=H\Psi_2$. By Lemma \ref{SymFock} the left inner/right outer factorization has the same outer factor $H$, hence there are left inner factors $\Phi_1$ and $\Phi_2$ such that $F=\Phi_1 H$ and $G=\Phi_2 H$.
Then as in the proof of Theorem \ref{subinner/free outer} it follows that $h=H\circ U$ is the free outer factor of both $f$ and $g$. \end{proof}

The uniqueness part of the following Theorem can be understood to be a more general version of the fact that solutions to extremal classical Nevanlinna-Pick problems are unique and given by Blaschke products, and of a generalization thereof by Sarason \cite[Proposition 5.1]{Sarason67}; also see Theorem \ref{ExtremalPick}.
\begin{theorem} \label{uniqueLiftofSubinner} If  $\varphi \in \Mult(\HH_k)$ is subinner, then there is a unique $\Phi\in \MFockL$ such $\|\Phi\|_\infty =1$ and $\Phi\circ u=\varphi$. Moreover,  $\Phi$ is necessarily left inner.
\end{theorem}
\begin{proof} The existence of $\Phi$ is a consequence of the lifting theorem, which was mentioned in the introductory remarks of this section.

  Since $\varphi$ is subinner, we may choose  $0\ne h\in \HH_k$ with $\|\varphi h\|=\|h\|$. By Corollary \ref{subinner has free outer} we may assume that $h$ is free outer and we also assume $h(x_0)>0$. Set $f=\varphi h$, $F=Uf$, and $H=Uh\in \HH$. Since  $h$ is free outer in $\HH_k$, it follows that   $H$ must be $\HH$-free outer by Lemma \ref{lem:free_outer_H}.  Thus by Theorem \ref{free-outer-in*invariant} $H$ is right outer in $\Fock$. Also $H(0)=H(u(x_0))=h(x_0)>0$.

Let $\Phi\in \MFockL$ be any multiplier satisfying $\|\Phi\|_\infty =1$ and $\Phi\circ u=\varphi$.
The properties of $\Phi$ imply that $UM_\varphi=P_\HH L_\Phi U$, and hence $U(\varphi h) = P_{\HH}\Phi H$. We have $$\|H\|= \|h\|=\|\varphi h\|=\|P_{\HH}\Phi H\| \le \|\Phi H\|\le \|H\|.$$
Thus we must have $ \|P_{\HH}\Phi H\|= \|H\|$, and we conclude from  Theorem \ref{inner-in*invariant} that $\Phi$ is left inner. Furthermore, the equality $\|P_{\HH}\Phi H\| = \|\Phi H\|$ implies that $ \Phi H \in \HH$ and hence $F=U(\varphi h)=\Phi H$.

Thus, $\Phi$ and $H$ are  the unique left inner/right outer factors of the function $F=\Phi H\in \Fock$. In particular, $\Phi$ was determined uniquely.
\end{proof}
\begin{theorem}\label{freeOuterIsCyclic} If $f\in \HH_k$ is a free outer function, then $f$ is cyclic in $\HH_k$.\end{theorem}
\begin{proof} This follows from Theorem \ref{free-outer-in*invariant}  with arguments as used in the previous proofs. Indeed, if  $f\in \HH_k$ is free outer, then $f=F\circ u$ for  $F=Uf\in \HH$. Then $F$ must be $\mathcal{H}$-free outer in $\HH=\ran U$ by Lemma \ref{lem:free_outer_H}.
  Theorem \ref{free-outer-in*invariant} implies that $F$ is right outer in $\mathcal{F}^2_d$. Hence there are $H_n\in \mathcal{F}^{\infty,r}_d$ such that $R_{H_n} F\to 1$. Set $h_n(x)=H_n(u(x))$, then as explained in the beginning of this section $h_n \in \Mult(\HH_k)$ and $\|h_n f-1\|_{\HH_k} \le \|R_{H_n}F-1\|_{\mathcal{F}^2_d}\to 0$. Since $\Mult(\mathcal{H}_k)$ is dense in $\mathcal{H}$ (see also Lemma \ref{Lemma7.2} below), $f$ is cyclic in $\HH_k$.
\end{proof}

\section{The Sarason function}\label{SecSarasonF}

In this section we will prove Theorems \ref{SarasonFunction} and \ref{H2 interpolating}. We start with some elementary facts about Pick kernels. The first two lemmas are versions of \cite[Prop. 4.4]{Serra} and its proof. For completeness we record the short proofs. Let $X$ be a set and $w_0\in X$, $\HK$ be a Hilbert space, $u: X \to \HK$ be a function with $u(w_0)=0$, and let $$k_w(z)= \frac{1}{1-\la u(z),u(w)\ra_\HK}$$ be a normalized Pick kernel that is the reproducing kernel for the Hilbert function space $\HH$. Of course, in order for this to be well-defined we need to assume $\|u(z)\|<1$ for all $z\in X$.

\begin{lemma}\label{bMult} If $x\in \HK$, then $\varphi_x(z)=\la u(z),x\ra_{\HK}$ defines a  multiplier on $\HH$ with $\|\varphi_x\|_{\MuH}\le \|x\|$.
\end{lemma}

\begin{proof} There is nothing to prove if $x=0$. Hence it will suffice to prove the Lemma for $\|x\|=1$. Let $P$ denote the projection of $\HK$ onto $x^\perp$, then $\la u(z), u(w)\ra_{\HK} =\varphi_x(z) \overline{\varphi_x(w)} +\la P u(z),u(w)\ra_{\HK}$ and hence
  $$(1-\varphi_x(z) \overline{\varphi_x(w)})k_w(z)= 1+\la P u(z),u(w)\ra_{\HK} k_w(z).$$ By the Schur product theorem this is positive definite, hence $\varphi_x$ must be a contractive multiplier on $\HH$; see for example \cite[Corollary 2.37]{AgMcC}.
\end{proof}

\begin{lemma}\label{Lemma7.2} If $w\in X$, then $k_w \in \Mult(\mathcal{H})$ and $\|k_w\|_{\MuH}\le 2\|k_w\|^2$.\end{lemma}
\begin{proof} If $w\in X$, then by the previous lemma $\|\varphi_{u(w)}\|_{\MuH} \le \|u(w)\|<1$. Note that $k_w=\frac{1}{1-\varphi_{u(w)}}$. Hence
  \begin{equation*}
\|k_w\|_{\MuH} \le \sum_{n=0}^\infty \|\varphi_{u(w)}\|^n_{\MuH}\le \frac{1}{1-\|u(w)\|_{\HK}}\le 2\|k_w\|^2_{\HH}.
\qedhere
  \end{equation*}
\end{proof}
\begin{lemma} \label{continuityLemma}  If $z, z_n\in X$ such that $k_{z_n}\to k_z$ in $\HH$, then $k_{z_n}\to k_z$ in $\MuH$.
\end{lemma}
\begin{proof} Note that $k_z-k_w= \varphi_{(u(z)-u(w))}k_zk_w$, hence by Lemmata \ref{bMult} and \ref{Lemma7.2},
$$\|k_z-k_w\|_{\MuH} \le 4\|k_z\|^2\|k_w\|^2 \|u(z)-u(w)\|_{\HK}.$$
Now suppose $k_{z_n} \to k_z$ in $\HH$. Then $k_{z_n}(z_n)=\|k_{z_n}\|^2\to \|k_z\|^2=k_z(z)$ and $k_{z_n}(z) \to k_z(z)$. We know that $k_x(y)=\frac{1}{1-\la u(y),u(x)\ra_\HK}$ for all $x,y\in X$, thus $\|u(z_n)\|^2_\HK\to \|u(z)\|^2_\HK$ and $\mathrm{Re}\la u(z), u(z_n)\ra_\HK \to \|u(z)\|^2_\HK$. Hence $$\|u(z_n)-u(z)\|^2_{\HK}= \|u(z_n)\|^2_{\HK}+ \|u(z)\|^2_{\HK}- 2 \mathrm{Re}\la u(z), u(z_n)\ra_\HK  \to 0$$ and the lemma follows.
\end{proof}

\begin{proof}[Proof of Theorem \ref{SarasonFunction}] The equivalence of (i) and (ii) is Theorem \ref{Pf=Pg}, it was established in Section \ref{SecPickspaces}. (ii)$\Rightarrow$(iii) is trivial. We will now show  (iii)$\Rightarrow$(ii).

If $f,g \in \MuH$, then $V_f= M_f^*f\in \HH$ and $P_f(\varphi)= \la \varphi, M_f^* f\ra$ for all $\varphi\in \MuH$, and similarly for $g$. Thus, if $V_f=V_g$, then $M_f^*f=M_g^*g$ and hence $P_f=P_g$.

Next assume that  $d\in \N \cup \{\infty\}$, $X=\mathbb{B}_d$, and that we are given coefficients $\{a_\alpha\}$ such that for each $\alpha \in I_d$ $a_\alpha >0$ and $\sum_{\alpha\in I_d} a_\alpha |z^\alpha|^2<\infty$ for each $z\in \Bd$. Then an application of the Cauchy-Schwarz inequality implies that the series
$$k_w(z)=\sum_{\alpha\in I_d} a_\alpha z^\alpha \overline{w}^\alpha$$ converges absolutely for all $z,w\in \mathbb{B}_d$. Then $\{f_\alpha(z)=\sqrt{a_\alpha} z^\alpha: \alpha \in I_d\}$ is an orthonormal basis for $\HH$ and every $f\in \HH$ is of the form $f=\sum_{\alpha\in I_d} \la f,f_\alpha\ra f_\alpha$, hence $$f(z)=\la f, k_z\ra =\sum_{\alpha\in I_d} \hat{f}(\alpha) z^\alpha,$$ where $\hat{f}(\alpha) =\la f,f_\alpha\ra \sqrt{a_\alpha}$, and the series converges absolutely for each $z\in \Bd$. Furthermore, we have $$\|f\|^2=\sum_{\alpha\in I_d}| \la f,f_\alpha\ra|^2=\sum_{\alpha\in I_d}\frac{|\hat{f}(\alpha)|^2}{a_\alpha}.$$

 Let $\HL$ denote the weak*-closure of the finite linear combinations of reproducing kernels $k_w$ in $\MuH$. If $f,g \in \HH$ such that $V_f=V_g$, then $P_f$ and $P_g$ agree on $\HL$. We will show that $\HL =\MuH$.

We will first show that $\mathcal{L}$ contains all monomials.
Let $\alpha \in I_d$. Then there is  $n\in \N$ such that $\alpha \in I_n$. If $n<d$, then we identify $\mathbb{B}_n$ with $\{z\in \Bd: z_j=0 \text{ for all } j>n\}\subseteq \Bd$. Note that if $w\in \mathbb{B}_n$, then for all $z\in \Bd$ we have $k_w(z)=\sum_{\beta\in I_n} a_\beta \overline{w}^\beta z^\beta$.
A standard result about power series  shows that the last series converges uniformly
on compact subsets of $\mathbb{B}_n \times \mathbb{B}_n$; see, for instance, \cite[Corollary 1.16]{Range86}.
Therefore the function $(z,w) \mapsto k_w(z)$ is continuous on $\mathbb{B}_n \times \mathbb{B}_n$.
Since
\begin{equation*}
  \|k_w - k_z\|^2 = k_w(w) - 2 \operatorname{Re} k_w(z) + k_z(z),
\end{equation*}
we see that $w \mapsto k_w$ is continuous as a function $\mathbb{B}_n \to \mathcal{H}$,
and hence as a function $\mathbb{B}_n \to \Mult(\mathcal{H})$ by Lemma \ref{continuityLemma}.

Next, let $\sigma$ be the rotationally invariant Borel probability measure on $\mathbb{B}_n$ and let $ 0 <r <1$.
For fixed $z \in \mathbb{B}_d$, the series $k_w(z) = \sum_{\beta \in I_n} a_\beta \overline{w}^\beta z^\beta$,
regarded as a function of $w$, converges uniformly on compact subsets of $\mathbb{B}_n$.
Hence
\begin{equation*}
  \int_{\partial \mathbb{B}_n} w^\alpha k_{r w}(z) d \sigma(w) = a_\alpha r^{|\alpha|} \int_{\partial \mathbb{B}_n} |w^\alpha|^2 \, d \sigma(w) z^\alpha.
\end{equation*}
On the other hand, the integral $\int_{\partial \mathbb{B}_n} w^\alpha k_{ rw} d \sigma(w)$
converges in multiplier norm and in particular belongs to $\mathcal{L}$.
This shows that $z^\alpha \in \mathcal{L}$, and so $\mathcal{L}$ contains all polynomials.

If $d<\infty$, then the proof of Lemma 4.1 of \cite{GrRiSu} applies to show that the polynomials are weak*-dense in $\MuH$, hence $\HL=\MuH$ in that case. We note here that
the proof of Lemma 4.1 of \cite{GrRiSu} shows that every multiplier is a weak operator topology
limit of a bounded sequence of polynomials; moreover the weak operator and weak* topologies agree on bounded
subsets of $B(\mathcal{H})$.

We finish the proof by showing that the polynomials are also weak* dense in case $d= \infty$.
If $z\in \mathbb{B}_\infty$, $t\in \R$, and $n \in \N$, then let $$z_{t,n}=(z_1, \dots, z_n, e^{it}z_{n+1}, e^{it}z_{n+2}, \dots).$$
Given $\varphi \in \Mult(\mathcal{H})$, let $\varphi_{t,n} (z) = \varphi(z_{t,n})$.
Since $k_{w_{t,n}}(z_{t,n}) = k_w(z)$, it follows that
$\varphi_{t,n} \in \Mult(\mathcal{H})$ and $\|\varphi_{t,n}\|_{\Mult(\mathcal{H})} = \|\varphi\|_{\Mult(\mathcal{H})}$
for all $n \in \mathbb{N}$ and $t \in \mathbb{R}$.
Moreover, absolute convergence of the power series expansion of $\varphi$ and the Weierstra\ss\ M-test
show that for each $z \in \mathbb{B}_\infty$ and $n \in \mathbb{N}$,
the  power series of $\varphi_{t,n}(z)$ converges uniformly in $t \in \mathbb{R}$, and that the
map $t \mapsto \varphi_{t,n}(z)$
is continuous. Hence $t \mapsto \varphi_{t,n}$ is weak*-continuous as a map $\mathbb{R} \to \Mult(\mathcal{H})$.

Uniform convergence shows that for each $z \in \mathbb{B}_\infty$,
\begin{equation*}
  \int_{0}^{2 \pi} \varphi_{t,n}(z) \frac{dt}{2 \pi} = \varphi(z_1, \ldots,z_n, 0, 0, \ldots).
\end{equation*}
Thus, if $\varphi_n(z) = \varphi(z_1,\ldots,z_n,0, 0, \ldots)$, then
$\varphi_n = \int_{0}^{2 \pi} \varphi_{t,n} \frac{dt}{2 \pi}$, and the integral converges in the weak* topology.
Hence $\|\varphi_n\|_{\Mult(\mathcal{H})} \le \|\varphi\|_{\Mult(\mathcal{H})}$.
Moreover, $\varphi_n$ converges to $\varphi$ in the norm of $\mathcal{H}$; hence $\varphi_n$ converges to $\varphi$
pointwise on $\mathbb{B}_\infty$ and thus in the weak* topology of $\Mult(\mathcal{H})$.

It remains to show that each $\varphi_n$ belongs to the weak* closure of the polynomials.
Let $\psi = \varphi_n$ for some $n \in \mathbb{N}$.
Then $\psi_{t,1}(z) = \psi(e^{ i t} z)$, and as observed above, the map $t \mapsto \psi_{t,1}$
is weak* continuous.
In this setting, the argument in the proof of Lemma 4.1 of \cite{GrRiSu}
shows that the Fej\'er means of $\psi$, which are polynomials, converge to $\psi$ in the weak* topology.
\end{proof}

\begin{proof}[Proof of Theorem \ref{H2 interpolating}] For $z, \lambda\in \D$ let $k_\lambda(z)=\frac{1}{1-\overline{\lambda}z}$ denote the Szeg\H{o} kernel. As before let $\HL$ denote the weak*-closure of the finite linear combinations of reproducing kernels of $\mathcal{H} = H^2(\mathbb{D}) |_X$ in $\MuH$.
  Assume that $\lambda_1=0$.  For  $j\in \N$ we define $\psi_j:X\to \C$  to be 1 at $\lambda_j$ and 0 at all the other points of $X$. We will first show that each $\psi_j\in \HL\subseteq \MuH$.

Fix $j\in \N$, and let $N > j$. Then since the functions $k_{\lambda_n}|X$ are linearly independent there is a function $f_N=\sum_{n=1}^N a_n k_{\lambda_n}$ that equals 1 at $\lambda_j$ and $0$ at $\lambda_n$ for all $1\le n\le N, n\ne j$. Then $f_N$ is a rational function with simple poles at $1/{\overline{\lambda}_n}$ for $n=2, \dots, N$ and the numerator of $f_N$ has degree at most $N-1$ and vanishes at $\lambda_n$ for $1\le n\le N, n\ne j$.
Note that $a_n \neq 0$ for each $n$, since the numerator of $f_N$ must have $N-1$ zeros.
Hence $$f_N(z)= \frac{1-|\lambda_j|^2}{1-\overline{\lambda}_jz}\prod_{n=1, n\ne j}^N \frac{b_n(z)}{b_n(\lambda_j)}$$ where $b_n(z)= \frac{\lambda_n-z}{1-\overline{\lambda}_nz}$. Then $f_N \in \HL$.
(Alternatively, one could define $f_N$ by the last formula and observe that $f_N \in \mathcal{L}$ by partial
fraction decomposition.) Moreover,

$$\|f_N\|_{\MuH} \le \|f_N\|_{H^\infty} \le \frac{2}{\prod_{n=1, n\ne j}^\infty |b_n(\lambda_j)|}<\infty,$$ and $f_N(\lambda_n)\to \delta_{nj}$ for each $n \in \N$. This shows that $f_N \to \psi_j$ in the weak*-topology of $\MuH$. Hence $\psi_j\in \HL$.

If $\varphi\in H^\infty$, then $\varphi|X\in \MuH$ with $\|\varphi\|_{\MuH}\le \|\varphi\|_{H^\infty}$. Conversely, by the commutant lifting theorem every element in $\MuH$ is of the form $\varphi|X$ for some $\varphi\in H^\infty$ with $\|\varphi\|_{\MuH}= \|\varphi\|_{H^\infty}$.

For $N\in \N$ set $B_N(z) =\prod_{n>N} \frac{\overline{\lambda}_n}{|\lambda_n|}b_n(z)$. Then for every $\varphi \in H^\infty$ we note that $B_N \varphi \to \varphi$ in the weak*-topology of $H^\infty$. This implies that $B_N \varphi|X \to \varphi|X$  in the weak*-topology of $\MuH$. But we have \begin{equation*}
(B_N \varphi)|X =\sum_{j=1}^N B_N(\lambda_j)\varphi(\lambda_j) \psi_j\in \HL \ \text{ for each }N. \qedhere\end{equation*}
\end{proof}

\section{Approximation by subinner functions}\label{SecApproxBySubinner}

In every reproducing kernel Hilbert space $\HH_k$ the set of finite linear combinations of reproducing kernels is dense. If $\HH_k$ is separable, then there is a countable set $Y\subseteq X$ such that the set of finite linear combinations of $\{k_x\}_{x\in Y}$ is dense in $\HH_k$.
Indeed, since subspaces of separable metric spaces are separable, there exists a countable set $Y \subseteq X$ so
that $\{k_x: x \in Y\}$ is dense in $\{k_x: x \in X\}$, hence the set of finite linear combinations of $\{k_x\}_{x \in Y}$ is dense in $\HH_k$.
Let $Y=\{x_n\}\subseteq X$ be any such countable subset of a separable complete Pick space $\HH_k$, then $$\HM_n= \Big\{\sum_{j=1}^nc_jk_{x_j}: c_j\in \C \Big\}$$ defines a nested sequence of finite dimensional $*$-invariant subspaces of $\HH_k$ whose union is dense in $\HH_k$.

Next we record a simple lemma.

\begin{lemma} \label{subinnerMatch} Let $k$ be a normalized complete Pick kernel and let $\HM\subseteq \HH_k$ be a finite dimensional $*$-invariant subspace.

If $\varphi\in \Mult (\HH_k)$ with $\|M_\varphi^*|\HM\|=1$, then there is a subinner function $\psi$ such that $\varphi-\psi \in \HM^\perp$. Moreover,
there exist $f, g \in \mathcal{M}$ with $g = \psi f$ and $\|f\|=\|g\|\ne 0$. \end{lemma}
\begin{proof} Let $T=M_\varphi^*|\HM$. Then $T:\HM\to \HM$, $\|T\|=1$, and $TM_u^*=M_u^*T$ for all $u \in \Mult(\HH_k)$. Hence by the Ball-Trent-Vinnikov Commutant Lifting Theorem, \cite{BallTrentVinnikov}, there is a multiplier $\psi\in \Mult(\HH_k)$ with $\|\psi\|_{\Mult(\HH_k)}=1$ and $T=M_\psi^*|\HM$. The identity $M_\psi^*|\HM=M_\varphi^*|\HM$ implies $\ran M_{(\varphi-\psi)} \subseteq \HM^\perp$.

Furthermore, since $\HM$ is finite dimensional $M_\psi^*|\HM$ attains its norm. Hence there is $g\in \HM$ with $\|M_\psi^* g\|=\|g\|=1$. Then  by Lemma \ref{contractionLemma} we have $M_\psi M_\psi^*g=g$. Lemma \ref{subinnerMatch} evidently follows with $f= M_\psi^*g\in \HM$.
\end{proof}

If $\HH_k$ is a finite dimensional complete Pick space, then the subinner functions are precisely the multipliers of norm 1,
as operators on finite dimensional spaces attain their norm.
We will now prove that in infinite dimensional complete Pick spaces, every function in the unit ball of the multiplier algebra is a weak* limit of subinner functions.
\begin{theorem} Let $k$ be a normalized complete Pick kernel on a set $X$, and suppose $\HH_k$ is separable and infinite dimensional.

Then for every $\varphi \in \Mult(\HH_k)$ with $\|\varphi\|_{\Mult(\HH_k)}\le 1$, there is a sequence of subinner functions $\varphi_n$ such that $\varphi_n(x)\to \varphi(x)$ for every $x\in X$.

Moreover, if we are given a nested sequence of finite dimensional $*$-invariant subspaces $\{\HM_n\}_{n\in \N}$ such that
\begin{enumerate}
  \item $1\in\HM_n\subseteq \HM_{n+1} \subseteq \Mult(\mathcal{H}_k)$ for each $n$,
\item $\HH_k=\bigvee_{n=1}^\infty \HM_n$,
\end{enumerate}
 then for each $n$ the subinner function $\varphi_n$ can be chosen to satisfy $g_n = \varphi_n f_n$ for some functions $f_n,g_n \in \mathcal{M}_n$.
\end{theorem}
Taking $\mathcal{M}_n$ as at the beginning of the section, the approximating subinner functions can be thought of as ratios of finite linear combinations of reproducing kernels.
 If for $1\le d <\infty$, we have $\HH_k \subseteq \Hol(\Bd)$, the monomials $\{z^\alpha\}_{\alpha \in \N_0^d}$ form an orthogonal basis for $\HH_k$ and are multipliers of $\HH_k$, then one easily checks that
 one can take $\HM_n$ to be the span of  polynomials that are homogeneous of order $\le n$, and the approximating subinner functions are going to be rational functions. Thus, Theorem \ref{Caratheodory Approximation} follows.
\begin{proof}
  We first observe that there exist sequences $\{\mathcal{M}_n\}$ of subspaces as in the statement.
Let $x_0\in X$ be the normalization point of the reproducing kernel $k$, i.e. $k_{x_0}=1$, and choose
 $\{x_n\}_{n\ge 1}\subseteq X$ such that the linear span of $\{k_{x_n}\}_{n \ge 0}$ is dense in $\HH_k$.
  For $n \ge 0$ set $\HM_n =\{\sum_{j=0}^n c_j k_{x_j}: c_0, \dots, c_n\in \C\}$, then $1\in\HM_n$ for each $n$ and the $\HM_n$'s form a nested sequence of finite dimensional $*$-invariant subspaces such that $\HH_k=\bigvee_{n=1}^\infty \HM_n$.
  Moreover, by Lemma \ref{Lemma7.2}, we have $\mathcal{M}_n \subseteq \Mult(\HH_k)$.

Next let $\varphi \in \Mult(\HH_k)$ with $\|\varphi\|_{\Mult(\HH_k)}\le 1$ and fix a sequence of subspaces $\HM_n$ satisfying the conditions (i) and (ii) of the Theorem. Clearly, we may assume that $\HM_{n-1} \subsetneq \HM_{n}$ for all $n$,
so there is $\psi_n \in \HM_n \ominus \HM_{n-1}$ with $\|\psi_n\| = 1$. Then $\{\psi_n\}$ is a sequence of multipliers
that is orthonormal in $\mathcal{H}_k$.
Hence $\{\psi_n\}$ converges to $0$ weakly and so $\psi_n(x) \to 0$ for each $x\in X$.
  Write $P_n$ for the projection of $\HH_k$ onto $\HM_n$. Then $I-P_n\to 0$ in the strong operator topology.

  Fix $n \ge 1$. For $t\ge 0$ set $\varphi_{n,t} = \varphi +t \psi_n$. Then if $t\ge 2$ we have
  \begin{align*}\|M^*_{\varphi_{n,t}}|\HM_n\|=\|P_{n}M_{\varphi_{n,t}}|\HM_n\| &\ge \|P_{n}(\varphi+t\psi_n) 1\|_{\HH_k} \\ &\ge t\|\psi_n\|_{\HH_k}-1 \ge 1.\end{align*} Here we used that $1\in \HM_n$. We also have
  $$\|M^*_{\varphi_{n,0}}|\HM_n\|=\|M^*_{\varphi}|\HM_n\| \le \|\varphi\|_{\Mult(\HH_k)}\le 1.$$ Thus, by continuity  there is $0\le t_n \le 2$ such that $\|M^*_{\varphi_{n,t_n}}|\HM_n\|=1$. Now we apply Lemma \ref{subinnerMatch} to conclude that there is a subinner function $\varphi_n$, satisfying $g_n = \varphi_n f_n$ for some $f_n,g_n \in \mathcal{M}_n$, such that $\varphi_{n,t_n}-\varphi_n\perp \HM_n$.

  Then $\varphi_n-\varphi= (I-P_n)(\varphi_n-\varphi_{n,t_n})+t_n \psi_n$ and hence
  for all $x\in X$ we have
  \begin{align*}|\varphi_n(x)-\varphi(x)|   &\le |\la \varphi_n-\varphi_{n,t_n}, (I-P_n)k_x\ra|+2|\psi_n(x)|\\
  &\le (\|\varphi_n\|+\|\varphi_{n,t_n}\|)\|(I-P_n)k_x\|+2|\psi_n(x)|\\
  &\le 3\|(I-P_n)k_x\|+2|\psi_n(x)| \to 0 \qedhere\end{align*}
\end{proof}
\section{Solutions to extremal Pick problems}\label{SecPickProblem}

For $\lambda_1, \dots, \lambda_n \in X$ and $w_1, \dots, w_n \in \C$ let $K=(k_{\lambda_i}(\lambda_j))$ be the Gramian of the kernels, and let $P=\left( (1-\overline{w}_iw_j)k_{\lambda_i}(\lambda_j)\right)$ be the Pick matrix. The Pick problem with data  $\lambda_1, \dots, \lambda_n \in X$ and $w_1, \dots, w_n \in \C$ is to find a $\varphi \in \MuH$ with $\|\varphi\|_\MuH \le 1$ and $\varphi(\lambda_j)=w_j$ for $j=1, \dots, n$. It is a property of complete Pick kernels, that a Pick problem has a solution, if and only if the Pick matrix $P\ge 0$, i.e. $\la Pa,a\ra \ge 0$ for all $a\in \C^n$, see  \cite{McCullNePickInt}, \cite{Quiggin}, \cite{McCullCaraInt}; see also \cite{AgMcC}.

\begin{lemma} If $P\ge 0$, then the null space of $K$ is contained in the null space of $P$.\end{lemma}
 \begin{proof} Let $a=(a_1,\dots, a_n)\in \C^n$ with $Ka=0$. Set $f= \sum_{j=1}^n a_j k_{x_j}$ and $g=  \sum_{j=1}^n a_j \overline{w}_j k_{x_j}$. Then $\|f\|^2=\la Ka,a\ra_{\C^n}$ and $\|f\|^2-\|g\|^2=\la Pa,a\ra_{\C^n}$. Thus, $Ka=0$ implies that $\|f\|=0$, then the positivity of $P$ implies that $\|g\|=0$ as well and hence $\la Pa,a\ra_{\C^n}=0.$ This implies that $Pa=0$ since $P\ge 0$.
 \end{proof}

 Thus, if $P\ge 0$, then the rank of $P$ must be less than or equal to the rank of $K$. Of course, $K$ will be full rank unless
 the embedding $u$ from $X$ into ${\mathbb B}_d$  is not injective.
 \begin{definition} A Pick problem is called extremal, if $P\ge 0$ and $\rank \ P< \rank \ K$.
 \end{definition}

 It is known that solutions to extremal Pick problems are unique; see for instance \cite[Theorem 8.33]{AgMcC} for the case
 when $K$ has full rank.
\begin{theorem}\label{ExtremalPick} Solutions to extremal Pick problems are given by subinner functions. \end{theorem}
\begin{proof} The hypothesis implies that there is $a\in \C^n$ such that $Pa=0$ and $Ka\ne 0$. We use the same functions $f$ and $g$ as in the proof of the previous lemma, then $\|f\|\ne 0$ and $\|f\|=\|g\|$. If $\varphi$ is any solution to the Pick problem, then the equalities $M_\varphi^*k_{x_j}=\overline{w}_jk_{x_j}$ imply that $g=M_\varphi^* f$. By Lemma \ref{contractionLemma} with $T=M_\varphi^*$ we conclude that $f=\varphi g$. Hence $\varphi$ is subinner.
\end{proof}

Note that an extremal Pick problem can be regarded as a multiplier of $\mathcal{H} \big|_{ \{\lambda_1,\ldots,\lambda_n \}}$ of norm $1$,
which is subinner by finite dimensionality. This observation relates this section to Theorem \ref{uniqueLiftofSubinner}.

\section{The Dirichlet space}\label{SecDiriExample}
The Dirichlet space $D$ is defined as $$D=\{f\in \Hol(\D): \int_\D |f'(z)|^2 dA(z)<\infty\}$$ with norm given by $$\|f\|^2_D=\|f\|^2_{H^2} + \int_\D |f'|^2 \frac{dA}{\pi},$$
where $d A$ denotes integration with respect to planar Lebesgue measure.
Then, as mentioned in the Introduction, the reproducing kernel is $\frac{1}{\overline{w}z}\log\frac{1}{1-\overline{w}z}$ and it is well-known that this is a complete Pick kernel, see e.g. \cite[Corollary 7.41]{AgMcC}. We start this section by showing that  in $D$ a theorem like the Bergman inner-outer factorization that was mentioned in the Introduction does not hold.
\begin{theorem}\label{DiriNotFactor} There is a function $f\in D$ that is not of the form $f=\varphi h$, where $\varphi, h\in D$ and $\varphi$ is $D$-extremal, and $h$ is cyclic in $D$.
\end{theorem}
\begin{proof}
Recall from \cite{RiDiri} that every invariant subspace $\HM$ of the Dirichlet space $D$ is generated by $\mathcal{M} \ominus z \mathcal{M}$, and this space is
one dimensional. Hence $\HM$ is of the form $\HM=[\varphi]$ for some extremal function $\varphi$, and the function $\varphi$ is unique up to a multiplicative constant of modulus 1. In this case, we will say that $\varphi$ is the extremal function for $\HM$. Similarly, we will say that $\varphi$ is the extremal function for a function $f\in D$, if $[f]=[\varphi]$.
If $\varphi, h\in D$ where $\varphi$ is $D$-extremal and $h$ is cyclic in $D$, then since $\varphi$ is a contractive multiplier (\cite{RiSuTAMS}) one easily checks that $\varphi$ is the extremal function for $\varphi h$. Thus, in order to prove the Theorem it will suffice to show that there is an extremal function $\varphi$ and an analytic function $h$ such that $h \notin D$, but $f=h\varphi \in D$ with $[\varphi]=[f]$.

For $z\in \D$ consider the function $$g(z) =\log\frac{e}{1-z}= 1+\sum_{n=1}^\infty \frac{z^n}{n}.$$ Then there is a $C>0$ such that \begin{equation} \label{eqn:Dirichlet_est} \log \frac{e}{2}\le \log\frac{e}{|1-z|}\le |g(z)| \le C \log\frac{e}{|1-z|} \end{equation} and hence $g$ has an analytic square root, $h=\sqrt{g}$. Then $$|h'(z)|^2 \ge \frac{1}{4C \left(\log \frac{e}{|1-z|}\right)|1-z|^2}$$ and hence $h \notin D$, the Dirichlet space.

For $z\in \D$ and $|w|=1$ let $P_z(w)=\frac{1-|z|^2}{|1-\overline{w}z|^2}$ be the Poisson kernel for $\D$, and use
$$D_w(\varphi) = \int_{\D}|\varphi'(z)|^2 P_z(w) \frac{dA}{\pi} = \int_{|z|=1}\frac{|\varphi(w)-\varphi(z)|^2}{|w-z|^2}\frac{|dz|}{2\pi}$$ to denote the local Dirichlet integral of $\varphi\in D$, see \cite{RiSuMMJ}. Then, if $0<r< 1$ and $z\in \D$
\begin{align*}|h(rz)|^2&\le C\log\frac{e}{|1-rz|}\\
&=C\int_{|w|=1} \log \frac{e}{|1-rw|}\ P_z(w) \frac{|dw|}{2\pi}
\end{align*}
and hence
\begin{align*}
\int_{\D} |h(rz)\varphi'(z)|^2 \frac{dA}{\pi} &\le C\int_{|w|=1} \log \frac{e}{|1-rw|}\left(\int_{\D}|\varphi'(z)|^2 P_z(w) \frac{dA}{\pi}\right) \frac{|dw|}{2\pi}\\
&\le C\int_{|w|=1} \log \frac{e}{|1-rw|} D_w(\varphi) \frac{|dw|}{2\pi}\\
&\le K,
\end{align*}
where by Proposition 4.5 of \cite{RiYi} the constant $K$ is independent of $0<r<1$, whenever $\varphi$ is a Dirichlet extremal function. By Fatou's lemma this implies that $h\varphi'\in L^2_a$, the Bergman space.

Now let $\varphi$ be the extremal function for the invariant subspace $\HM=D\cap \psi H^2$, where $\psi(z)=e^{-\frac{1+z}{1-z}}$. Note that $\HM\ne (0)$ since one easily checks that $(1-z)^2\psi \in \HM$.
Then $\varphi= \psi\varphi_1$ for some $\varphi_1\in H^2$. By use of Carleson's formula for the Dirichlet integral, \cite{Ca_Diri} (also see \cite{RiSuMMJ} or Theorems 7.6.1 and 7.6.7 in \cite{EKM+14}),  we see that $\varphi_1 \in D$ and
$$\|\varphi\|^2_D= \int_{|z|=1}\frac{2}{|1-z|^2}|\varphi_1(z)|^2 \frac{|dz|}{2\pi}+\|\varphi_1\|^2_D.$$
Thus, $\varphi(z)/(1-z) \in H^2\subseteq L^2_a$.

For $z\in \D$ and $0<r<1$ we have from \eqref{eqn:Dirichlet_est} that
$$|h'(rz)|^2 \le \frac{1}{4 \log\frac{e}{2}}\frac{1}{|1-rz|^2}\le \frac{1}{\log\frac{e}{2}}\frac{1}{|1-z|^2}.$$
This implies that for all $0<r<1$
$$\|h_r'\varphi\|^2_{L^2_a} \le \frac{1}{\log\frac{e}{2}}\ \|\frac{\varphi}{1-z}\|^2_{L^2_a},$$ where $h_r(z)=h(rz)$.

Hence $f=h\varphi\in D$ since $(h\varphi)'=h'\varphi+h\varphi' \in L^2_a$. Furthermore, the above estimates show that $\|h_r\varphi\|_{D}$ is bounded independently of $r$ and hence $h_r\varphi\to f $ weakly in $D$. Thus  $f\in [\varphi]$. The inclusion $\varphi\in [f]$ follows from Corollary 5.5 of \cite{RiSuMMJ} since $|h|$ is bounded below on $\D$. Hence $[\varphi]=[f]$, i.e. $\varphi$ is the extremal function for $f$.
\end{proof}

Next we establish that subinner functions in the Dirichlet space in a certain sense differ drastically from their $H^2(\D)$-analogues.
\begin{theorem} \label{DiriDim=1} Let $D$ be the Dirichlet space. If $\varphi \in \Mult(D)$ is non-constant with $\|\varphi\|_{\Mult(D)}\le 1$, then  $\dim \ker (I-M_\varphi^*M_\varphi)\le 1$.
\end{theorem}
It follows from this, that if $\varphi$ is subinner in the Dirichlet space, and if $\varphi=\varphi_1\varphi_2$ for some contractive multipliers $\varphi_1$ and $\varphi_2$, then one of the two factors must be constant. Indeed, let $0\ne f\in D$ be such that $\|f\|=\|\varphi f\|$, then $$\|f\|=\|\varphi_1\varphi_2f\|\le \|\varphi_2 f\|\le \|f\|.$$ Thus, $\|\varphi_2 f\|=\|f\|$, and similarly $\|\varphi_1f\|=\|f\|$. Then if $\varphi_1$ is not a constant, then as $\varphi_2 f, f \in \ker (I-M_{\varphi_1}^*M_{\varphi_1})$, by the Theorem we must have that $f$ and $\varphi_2f$ are linearly dependent, hence $\varphi_2$ must be constant.

This property of Dirichlet subinner functions is even true for their unique lifts to left inner multipliers of $\Fock$: Let $u:\D\to \mathrm{ball}(\ell^2)$ a function satisfying $\frac{1}{z\overline{w}}\log \frac{1}{1-z\overline{w}}= \frac{1}{1-\la u(z), u(w)\ra_{\ell^2}}$. Then as in Section \ref{SecPickspaces} $u$ induces an isometric mapping $U:D\to \Fock$ with $d=\infty$ such that $U^*=C_u$. We write $\HH_D$ for the range of $U$.

 Now, if $\varphi$ is subinner in the Dirichlet space, then in Theorem \ref{uniqueLiftofSubinner} we proved that there is a unique left inner multiplier $\Phi\in \MFockL$ such that $\varphi=\Phi \circ u$.
 If $\Phi=\Phi_1\Phi_2$ for  left inner multipliers $\Phi_1, \Phi_2$, let $\varphi_j=\Phi_j\circ u$ for $j=1,2$. Then each $\varphi_j$ defines a contractive multiplier of $D$ with $\varphi=\varphi_1\varphi_2$. By the above we conclude that one of the two factors is a constant of modulus 1. Thus its value at 0 has modulus 1, which implies that its lift has value of modulus 1 at the origin. Thus, one of the left inner factors $\Phi_1$ or $\Phi_2$ must be constant.

\begin{proof} Let $\varphi \in \Mult(D)$ be non-constant with $\|\varphi\|_{\Mult(D)}\le 1$. We will show that every nonzero function in $\ker (I-M_\varphi^*M_\varphi)$ must be free outer. Then, if $\ker (I-M_\varphi^*M_\varphi)$ contained two linearly independent functions $f$ and $g$, then some non-trivial linear combination of $f$ and $g$ must be zero at the origin. Hence $\ker (I-M_\varphi^*M_\varphi)$ would contain a nontrivial non-cyclic function. Then our claim will contradict Theorem \ref{freeOuterIsCyclic}.

Thus, let $0 \ne f\in \ker (I-M_\varphi^*M_\varphi)$.

Then $f$ has a subinner/free outer factorization $f= \psi g$, and as in the proof of Corollary \ref{subinner has free outer} we see that the free outer factor $g\in \ker (I-M_\varphi^*M_\varphi)$. Note that the associated vector states satisfy $P_f=P_g$. Under this hypothesis it is easy to see that Lemma 5.3 of \cite{LuoRi} implies that  $$\|pg\|^2_D-\|pf\|^2_D= \int_{|z|=1} D_z(p)(|g(z)|^2-|f(z)|^2) \frac{|dz|}{2\pi}$$ for every polynomial $p$. Here $D_z(p)$ denotes the local Dirichlet integral of $p$, $$D_z(p)=\int_{|w|=1}\left|\frac{p(z)-p(w)}{z-w}\right|^2 \frac{|dw|}{2\pi}.$$
Now let $p_n$ be a sequence of polynomials such that $p_ng\to \varphi g$ in the Dirichlet norm. Then $p_ng\to \varphi g$ in $H^2$ and by possibly considering a subsequence, we may assume that $p_n(w)\to \varphi(w)$ for a.e. $w\in \T$. Then by Fatou's lemma we have for a.e.\ $z\in \T$ that $D_z(\varphi) \le \liminf_{n\to \infty} D_z(p_n)$. Thus, substituting $f=\psi g$ and using Fatou's lemma again we obtain
\begin{align*}0 &\le \int_{|z|=1}D_z(\varphi)(1-|\psi|^2)|g|^2 \frac{|dz|}{2\pi}\\
&\le \liminf_{n\to \infty}  \int_{|z|=1} D_z(p_n)(1-|\psi|^2)|g|^2 \frac{|dz|}{2\pi}\\
&= \liminf_{n\to \infty} \|p_ng\|^2_D-\|p_nf\|^2_D\\
&= \|\varphi g\|^2_D-\|\varphi f\|^2_D=0
\end{align*}
Since $\varphi$ is not constant it follows that $D_z(\varphi)>0$ for a.e.\ $z\in \T$ and since $g\ne 0$ this implies that $|\psi(z)|=1$ for a.e. $z\in \T$. Thus, $\psi$ is an inner function, but it is well-known that the only inner functions that are contractive multipliers of the Dirichlet space are constant.
Indeed, this follows since $\|\psi\|_{H^2} = 1 \ge \|\psi\|_D$ for such $\psi$.
This implies that $f$ is free outer.
\end{proof}

\begin{example} \label{ExampleDirichlet} Let $\lambda \in \C$. We will determine the subinner/free outer factorization of $g(z)=z-\lambda$ in $D$.

  If $\lambda=0$, then $\varphi(z)= \frac{z}{\sqrt 2}$ is an extremal function in $D$, hence it is subinner with free outer factor $f(z)= \sqrt 2$. If $\lambda \ne 0$, then note that the Sarason function of $g$ equals $V_g(z)= 2+|\lambda|^2 -2\overline{\lambda}z$. Furthermore, we have $$[g]_*=\{a+bz: a, b\in \C\}.$$ Thus, the subinner/free outer factorization of $g$ is of the form $g= \varphi f$, where $f(z)=a+bz$ with $a>0$ and $V_f=V_g$. We calculate that the only possible $f$'s are $f(z)= \sqrt{2}-\frac{\overline{\lambda}}{\sqrt{2}}z$ and $f(z)=-\frac{\overline{\lambda}}{|\lambda|}g(z)$. Maximizing the value of $f(0)$ we conclude that $g$ is a free outer function, if and only if $|\lambda|\ge \sqrt{2}$. If $|\lambda|<\sqrt{2}$, then $\sqrt{2}-\frac{\overline{\lambda}}{\sqrt{2}}z$ is the free outer factor of $g$, and $\varphi(z)= \sqrt{2} \ \frac{z-\lambda}{2-\overline{\lambda}z}$ is the subinner factor.
\end{example}

\begin{theorem} There is a sequence of real numbers $R_n>1 $ such that whenever $p$ is a polynomial of degree $\le n$ that is free outer in $D$, then $p(z)\ne 0$ for all $|z|<R_n$.\end{theorem}
\begin{proof} It is known that the reproducing kernel of the Dirichlet space is of the form $k_w(z) =\frac{1}{1-\sum_{k=1}^\infty c_k \overline{w}^kz^k}$ for  coefficients that satisfy $c_k>0$ for all $k\ge 1$; this can be seen using a lemma of Kaluza, see Lemma 7.38 in \cite{AgMcC} and its proof. Since $k_z(z) \to \infty$ as $|z|\to 1$ we must have $\sum_{k=1}^\infty c_k=1$.  For $n \ge 1$ set $$k_w^n(z)= \frac{1}{1-\sum_{k=1}^n c_k \overline{w}^kz^k}.$$ It is easy to verify that $k^n_w(z) =\sum_{k=0}^\infty a_k(n) \overline{w}^kz^k$ where $a_k(n)=\frac{1}{k+1}$ for all $0\le k\le n$. $k^n$ is a normalized complete Pick kernel of a Hilbert function space $\HH_n$. For all polynomials of degree $\le n$ we have $\|p\|^2_D = \|p\|^2_{\HH_n}=\sum_{k=0}^n (k+1)|\hat{p}(k)|^2$. This implies that the Sarason functions $V_p$ of polynomials $p$ of degree $\le n$ are the same, whether they are computed with respect to $D$ or with respect to $\HH_n$. Since free outer polynomials $p$ are solutions to the extremal problem $\sup\{|q(0)|: V_q=V_p, q\in [p]_*\}$ it follows that such a polynomial of degree $\le n$ is free outer in $D$, if and only if it is free outer in $\HH_n$.

  $k^n$ extends to be defined in a disc of radius $R_n>1$, and hence any cyclic function in $\HH_n$ cannot have any zeros in the open disc of radius $R_n$. Since free outer functions must be cyclic (by Theorem \ref{freeOuterIsCyclic}), the result follows.
\end{proof}

\section{The Drury-Arveson space}\label{SectionDruryArveson}
The Dirichlet space embeds in the Drury-Arveson space $H^2_\infty$, see the discussion at the beginning of Section \ref{SecPickspaces}. Under this embedding the function $f(z)=z-1$ that is cyclic in $D$ can be seen to be  sent to the function $h(z_1,z_2,\dots)=\sqrt{2}z_1-1$, which is not cyclic in $H^2_\infty$. Then $h$ is not free outer in $H^2_\infty$ by Theorem \ref{freeOuterIsCyclic}, and hence $f$ cannot be free outer in $D$ by an argument similar to Lemma \ref{lem:free_outer_H}.  Thus, it is of interest to show that there are cyclic functions in $H^2_d$ that are not free outer.
\begin{example}\label{ExampleDruryArveson} If $d\ge 2$, the functions
 $f(z)=1\pm 2z_1z_2$ are cyclic in $H^2_d$, but not free outer. In fact, in the case of $f(z)=1+2z_1z_2$ we have
$$f(z)=\frac{1+2z_1z_2}{\sqrt{2}(1+z_1z_2)}
{\sqrt{2}(1+z_1z_2)}$$
where $g(z)=\sqrt{2}(1+z_1z_2)$ is free outer, and $\varphi(z)=\frac{1+2z_1z_2}{\sqrt{2}(1+z_1z_2)}$ is subinner. \end{example}

\begin{proof} Fix $d\ge 2$ and note that $f(z)\ne 0$ for $z\in \Bd$. Hence for $0\le r<1$  the functions $g_r(z)=\frac{1}{1+r2z_1z_2}$ are analytic in a neighborhood of $\overline{\Bd}$. Thus,  $g_r\in \Mult(H^2_d)$  and
$(g_rf)(z)=\frac{f(z)}{1+r2z_1z_2}\in [f]$ for each $0\le r<1$. Furthermore,
\begin{align*} \Big\|\frac{f(z)}{1+r2z_1z_2}\Big\|_{H^2_d}&=\Big\|1+(1-r)\frac{2z_1z_2}{1+r2z_1z_2} \Big\|_{H^2_d}\\ &\le 1+(1-r) \Big\|\frac{2z_1z_2}{1+r2z_1z_2} \Big\|_{H^2_d} \end{align*}
and by Stirling's formula,
\begin{align*}
  \Big\| \frac{2 z_1 z_2}{1 + r 2 z_1 z_2} \Big\|_{H^2_d}^2 = \sum_{n=0}^\infty r^{2 n} 4^{n + 1} \frac{((n+1)!)^2}{(2 n + 2)!}
  &\lesssim \sum_{n=0}^\infty r^{2 n} (n+1)^{1/2} \\
  &\le \sum_{n=0}^\infty r^{2 n} (n+1) \\
  &= \frac{1}{(1 - r^2)^2}.
\end{align*}
Hence $g_r f$ is bounded in $H^2_d$ and so $g_r f \to 1$ weakly and $f$ is cyclic in $H^2_d$.

Next we compute the subinner/free outer factorization of $f$:

 $[1+2z_1z_2]_*=\{h(z)=c_0+c_1z_1+c_2z_2+c_3z_1z_2: c_j\in \C\}$ and one calculates
\begin{align*}&\la h, k_wh\ra =\\ &= |c_0|^2  + |c_1|^2  +|c_2|^2 + \frac{|c_3|^2}{2}+(\overline{c_0}c_1+\frac{c_3\overline{c_2}}{2})w_1+(\overline{c_0}c_2+\frac{c_3\overline{c_1}}{2})w_2 + \overline{c_0}c_3w_1w_2.\end{align*}
So $\la h, k_wh\ra=\la f, k_wf\ra$ for all $w$, if and only if
\begin{align*} |c_0|^2  + |c_1|^2  +|c_2|^2 + \frac{|c_3|^2}{2}&= 3,\\
\overline{c_0}c_1+\frac{c_3\overline{c_2}}{2}&=0,\\
\overline{c_0}c_2+\frac{c_3\overline{c_1}}{2}&=0,\\
\overline{c_0}c_3&=2.\end{align*} One further computes that $c_1=c_2=0$ and $(c_0,c_3)= (1,2), (\sqrt{2}, \sqrt{2})$ are the only solutions with $c_0>0$.
This implies that
$$[1+2z_1z_2]_*\cap \mathcal E_f= \{\alpha f,\beta g: |\alpha|=|\beta|=1\}.$$
Since $g(0)>f(0)$ we conclude that $g$ must be the free outer factor of $f$, hence the function $\varphi$ as above must be the subinner factor of $f$.
\end{proof}

Thus, $1+z_1z_2$ and $1-z_1z_2$ are free outer.
We will now show that the product $f(z_1,z_2)=1-z_1^2z_2^2$ is free outer as well. Since $[1-z_1^2z_2^2]_*$ is 9-dimensional, a proof along the lines as above would seem to involve some lengthy calculations. The following different approach uses Theorem \ref{free-outer-in*invariant} and Lemma \ref{lem:free_outer_H}, which states that a function is free outer if and only if it is left or right outer in $\Fock$.

 Let $S=\{1122, 1212,1221,2112,2121,2211\}$, so the free version of $f$ is $F(x_1,x_2)=1-\frac{1}{6}\sum_{w\in S}x^w$.

  Let $\Phi = \frac{1}{6} \sum_{w \in S} x^w$. Then $\|R_\Phi\| \le 1$. If $R_\Phi^* H = H$ for some $H \in \Fock$,
  then $R_\Phi H = H$ by Lemma \ref{contractionLemma}, and a look
at power series expansions yields $H = 0$. Hence $I - R_\Phi^*$ is injective, so $R_F = I - R_\Phi$ has dense range. Hence $F$ is left outer.

\begin{example}\label{ProdNotFreeOuter} The functions $g_j(z_1,z_2)=(1-z_j)$ are free outer in $H^2_2$ for $j=1,2$, but the product $(1-z_1)(1-z_2)$ is not.
\end{example}
It is easy to see that each $g_j$ is free outer. For example, one  verifies that if $a>0$ and $b\in \C$ so that $g(z_1,z_2)=a+bz_1\in [g_1]_*$, then $g$ satisfies $P_g=P_{g_1}$ only if $g=g_1$.

In order to prove that $g(z_1,z_2)=(1-z_1)(1-z_2)$ is not free outer, we consider for $a>0$ the one-parameter family of functions
$$f_a(z_1,z_2)= a-\frac{3a}{1+2a^2}(z_1+z_2)+\frac{z_1z_2}{a}.$$ Note that $f_1=g$. We will show that there is $a>1$ such that $V_{f_a}=V_g$. Then $|f_a(0)|>|g(0)|$ and hence $g$ cannot be free outer.

Set $$r(a)=\|f_a\|^2=  a^2 + 18\frac{a^2}{(1+2a^2)^2}+\frac{1}{2a^2}.$$
Then one calculates that for every $a>0$ we have
$$V_{f_a}(z)=r(a)-3(z_1+z_2)+2z_1z_2.$$
Thus, we will have $V_{f_a}=V_g$, if and only if $r(a)=r(1)$. Since $r$ is continuous and $r(a)\to \infty$ as $a\to \infty$, the result follows from the fact that $r'(1)=-1/3<0$.

\section{Common free outer factors}\label{SecVectorvalued}
In the following the existence of a common factorization of a sequence of functions goes back to work of Jury and Martin \cite{JuMa2}.
\begin{theorem}\label{VectorFactorization} Let $\HH_k$ be a Hilbert function space with normalized complete Pick kernel $k$ and let $\{f_n\}\subseteq \HH_k$ with $0 \neq \sum_{n\ge 0}\|f_n\|^2<\infty$.

Then there is a free outer function $g\in \HH_k$ and a sequence of multipliers $\{\varphi_n\} \subseteq \MuH_k$ such that
\begin{enumerate}
\item $f_n=\varphi_n g$ for each $n$,
\item $\|g\|^2=\sum_n\|f_n\|^2$,
\item $\sum_n \|\varphi_n h\|^2 \le \|h\|^2 $ for each $h\in \HH$.
\end{enumerate}
The factorization is unique if we assume that $g(z_0)>0$. Here $z_0$ is the point of normalization of the reproducing kernel $k$.
\end{theorem}
\begin{proof} Theorem 1.1 of \cite{JuMa2} shows the existence of a common factorization $f_n= \varphi_n g$ that satisfies conditions (i),(ii), and (iii). By Theorem \ref{subinner/free outer} $g= \psi g_0$ for some free outer function $g_0$ with $\|g_0\|^2=\|g\|^2=\sum_n\|f_n\|^2$, and a subinner function $\psi$. Then $f_n = \tilde{\varphi}_n g_0$ with $\tilde{\varphi}_n=\varphi_n \psi$. Thus, this factorization satisfies (i) and (ii), and it satisfies (iii) since $\psi$ is a contractive multiplier.

  In order to prove the uniqueness of the factorization note that conditions (i),(ii),(iii) and Lemma \ref{contractionLemma} imply that $g=\sum_{n}M_{\varphi_n}^*M_{\varphi_n} g$ and hence
\begin{align}\label{Psum}\sum_{n} \la \psi f_n,f_n\ra =\sum_{n}\la \psi \varphi_n g, \varphi_ng\ra = \la \psi g,g\ra \text{ for every }\psi \in \MuH.\end{align} Thus by Theorem \ref{Pf=Pg}, if $g(z_0)>0$, then $g$ is the unique free outer function satisfying (\ref{Psum}). Since free outer functions are cyclic, we must have $g(x)\ne 0$ for all $x$. This implies that the $\varphi_n$'s are uniquely determined as well.
\end{proof}
We will call $g$ the common free outer factor of the functions $f_n$. Let $\tau c(\HH)$ denote the trace class operators in $\HB(\HH)$, and write  $\|L\|_{\tau c}$ for the trace norm of $L\in \tau c(\HH)$.

\begin{corollary} \label{A1(1)} Let $\HH$ be  a normalized complete Pick space.
 If $L\in \tau c(\HH)$, then there are $f,g \in \HH$ such that  $f $ is free outer,  $\|f\|\|g\|\le \|L\|_{\tau c}$, and  $\tr(LM_\varphi) = \la \varphi f,g\ra$ for all $\varphi \in \MuH$.
\end{corollary}
In particular, if $f_n \in \HH$ with $\sum_n \|f_n\|^2 <\infty$, then there is a free outer function $f$ such that $P_f(\varphi)=\la \varphi f,f\ra =\sum_n \la \varphi f_n, f_n\ra$ for all $\varphi \in \MuH$, and hence $V_f = \sum_n V_{f_n}$.

The Corollary implies that   $\MuH$ has what is known as property $\mathbb{A}_1(1)$. As mentioned in the Introduction the result is known and our proof is in principle the same as in \cite{DavidsonHamilton}.
But the proof shows that with the results and concepts from this paper one can give a proof in terms of factorizations in Hilbert function spaces.
\begin{proof} Since $L$ is a trace class operator, there are $f_n, g_n \in \HH$ with $\|L\|_{\tau c}=\sum_n \|f_n\|\|g_n\| $ and such that $\tr (LM_\varphi)=\sum_n \la \varphi f_n,g_n\ra $ for all  $\varphi\in \MuH$.

We may assume that $\|f_n\|=\|g_n\|$ for each $n$, then $\|L\|_{\tau c}=\sum_n \|f_n\|^2<\infty$. Let $f$ be the common free outer factor of the $f_n$'s, then $f_n =\varphi_nf $ for some multipliers $\varphi_n$ satisfying (iii) of Theorem \ref{VectorFactorization}, i.e. the $\varphi_n$'s define a contractive column operator. The adjoint of that column operator is contractive also,  hence $g= \sum_n M^*_{\varphi_n} g_n$ converges in $\HH$ and $\|g\|^2 \le \sum_n\|g_n\|^2$. Thus $\|f\|\|g\| \le \sum_n \|f_n\|\|g_n\|= \|L\|_{\tau c}$ and whenever $\varphi \in \MuH$ we have
$$\la \varphi f,g\ra= \sum_n \la  \varphi f, M^*_{\varphi_n}g_n\ra = \sum_n \la \varphi \varphi_n f,g_n\ra = \sum_n \la \varphi f_n,g_n\ra= \tr (LM_\varphi).$$
If $f_n = g_n$ for all $n$, then as in the preceding proof, $f = \sum_n M_{\varphi_n}^* f_n = g$.
\end{proof}
It is well known that
$${\mathcal A}=\{M_\varphi: \varphi \in \MuH\}$$ is WOT-closed and hence weak-$*$ closed in $\HB(\HH)$ and that $\MuH$ is isometrically isomorphic to the dual of $\tau c(\HH)/^\perp {\mathcal A}$. For $L \in \tau c(\HH)$ we let $[L]$ denote the equivalence class of $L$ in $\tau c(\HH)/^\perp {\mathcal A}$, then by the Corollary we have
\begin{align} \label{inf}\|[L]\|= \inf\{\|f\| \|g\|: [L]=[f \otimes g]\}.\end{align}
Here, we write $f \otimes g$ for the rank one operator $h \mapsto \langle h,g \rangle f$, and
$\tr( M_\varphi (f \otimes g)) = \langle \varphi f,g \rangle$.

\begin{theorem} \label{infimum} Let $\HH$ be a normalized complete Pick space, and let $L\in \tau c(\HH)$ such that $[L]\ne 0$. Then the infimum in (\ref{inf}) is attained, if and only if there is a subinner/free outer pair $(\varphi,f)$ such that $[L]=[f \otimes \varphi f]$.

 Furthermore, if this happens, then the weak*-continuous functional on $\MuH$ induced by $L$ attains its norm at the subinner function $\varphi$. \end{theorem}

\begin{proof} Let $L\in \tau c(\HH)$ and suppose that there are $f,g \in \HH$ such that $[L]=[f\otimes g]$ and $\|[L]\|=\|f\|\|g\| \ne 0$. Note that if $f$ is not free outer, then $f=\psi f_0$ for some subinner/free outer pair $(\psi, f_0)$. Then for all $u \in \MuH$ we have $\la uf,g\ra=\la uf_0, M_\psi^* g\ra$ and $\|f_0\|\|M^*_\psi g\|\le \|f\|\|g\|$. Thus, without loss of generality we  may assume that $f$ is free outer with $f(z_0)>0$, and we will also  assume that $\|f\|=\|g\|$.

Then by the Hahn-Banach Theorem there is $\varphi\in \MuH$ with $\|\varphi\|_{\MuH}=1$ and
$$\|f\|^2=\|f\|\|g\|=\la \varphi f,g\ra \le \|\varphi f\|\|g\|\le \|f\|\|g\|=\|f\|^2.$$ Thus, we have equality in the Cauchy-Schwarz inequality, hence $\varphi f =\alpha g$ for some $\alpha \in \C$. It is easily seen that  $\alpha=1$. Hence $[L]=[f \otimes \varphi f]$, and the equality above implies $\|f\|=\|\varphi f\|$, i.e. $(\varphi,f)$ is a subinner/free outer pair. Furthermore, we see that the functional induced by $[L]$ attains its norm at $\varphi$.

The proof of the converse is trivial.
\end{proof}

It turns out that there is a modified subinner/free outer factorization for functions in all spaces whose reproducing kernel has a complete Pick factor.

\begin{corollary} \label{SpacesWithPickFactor} Let $\HH_k$ and $\HH_s$ be separable Hilbert function spaces on a set $X$ such that $s$ is a  complete Pick kernel, which is normalized at $z_0\in X$, and  such that $k/s $ is positive definite.

  If $f\in \HH_k \setminus \{0\}$, then there is a unique pair of functions $\varphi\in \Mult(\HH_s,\HH_k)$ and $g\in \HH_s$ such that
\begin{enumerate}
\item $f=\varphi g$,
\item $\|f\|_{\HH_k}=\|g\|_{\HH_s}$,
\item $g$ is free outer with $g(z_0)>0$,
\item  $\|\varphi\|_{\Mult(\HH_s,\HH_k)}\le 1$.
\end{enumerate}
\end{corollary}

\begin{proof} The condition that $k/s$ is positive definite implies that there is a sequence of functions $\{g_n\}$ on $X$ such that $k_w(z)=\sum_{n\ge 0} g_n(z) \overline{g_n(w)}s_w(z)$. Then the operator $T: \HH_k \to \bigoplus_{n \ge 0} \HH_s, k_w\to \{\overline{g_n(w)}s_w\}$ extends to be isometric. Furthermore, $T^*: \bigoplus_{n\ge 0} \HH_s \to \HH_k, T^*\{h_n\}=\sum_{n\ge 0} g_n h_n$ is a contraction.

  Let $f\in \HH_k$ with $f\ne 0$. Then the functions $f_n\in \HH_s$ defined by $Tf=\{f_n\}$ will satisfy $\|f\|^2_{\HH_k}= \sum_n \|f_n\|^2_{\HH_s}$ and $f=T^*Tf= \sum_{n\ge 0} g_nf_n$. Now let $g$ be the common outer factor of the $f_n$'s, $\varphi_n= f_n/g$, and $\varphi= \sum_{n\ge 0} g_n \varphi_n$. Then $\|g\|_{\HH_s}=\|f\|_{\HH_k}$, and by the fact that $T^*$ is a contraction and (iii) of Theorem \ref{VectorFactorization}, we have for all $h\in \HH_s$
$$\|\varphi h\|^2_{\HH_k} = \|\sum_{n\ge 0} g_n \varphi_n h\|^2_{\HH_k} \le \sum_{n\ge 0} \|\varphi_n h\|^2_{\HH_s}\le \|h\|^2_{\HH_s}.$$ Thus, $\varphi$ is a contractive multiplier from $\HH_s$ to $\HH_k$.

The uniqueness follows along similar lines as before. If $f=\varphi g=\varphi_0 g_0$ for free outer functions $g_0,g\in \HH_s$ with $\|f\|_{\mathcal{H}_k} = \|g\|_{\mathcal{H}_s} = \|g_0\|_{\mathcal{H}_s}$ and contractive multipliers $\varphi, \varphi_0$, then $M_\varphi^*f=g$. Hence  for all $\psi \in \Mult(\HH_s)$ we have
$$\la \psi g,g\ra_{\HH_s}= \la \psi g,M_\varphi^*f\ra_{\HH_s}=\la \psi f,f\ra_{\HH_k}.$$ Similarly for $\varphi_0$ and $g_0$, hence $\la \psi g,g\ra_{\HH_s}=\la \psi g_0,g_0\ra_{\HH_s}$, and $g=g_0$ now follows from Theorem \ref{Pf=Pg}. Hence $\varphi=\varphi_0$ since free outer functions have no zeros.
\end{proof}

For certain spaces on the disc the previous Corollary can be proved directly and this proof extends to cover certain Banach spaces. Let $1\le t<\infty$ and let $\mu$ be a positive Borel measure that is supported in the closed unit disc. Then $P^t(\mu)$ denotes the closure of the polynomials in $L^t(\mu)$.
For background on these spaces  see \cite{AlRiSuPtmu}.

\begin{theorem}\label{Ptmu} Let $1\le t<\infty$, suppose that  every $\lambda \in \D$ is a bounded point evaluation for $P^t(\mu)$, and that $P^t(\mu)$ is irreducible, meaning that $P^t(\mu)$ contains no non-trivial characteristic functions.

  Then every  $g\in P^t(\mu) \setminus \{0\}$ has a factorization $g=f \varphi $ such that
\begin{enumerate}
\item $f\in H^t(\T)$ is an outer function with $\|g\|_{P^t(\mu)}=\|f\|_{H^t(\T)}$, and
\item $\|u \varphi\|_{P^t(\mu)} \le \|u\|_{H^t(\T)}$ for every $u \in {H^t(\T)}.$
\end{enumerate}
If we also assume $f(0)>0$, then the two conditions determine $f$ and $\varphi$ uniquely.
\end{theorem}
Thus, $M_\varphi$ is a contractive multiplier from $H^t(\T)$ into $P^t(\mu)$ that attains its norm at the function $f$, so $\varphi$ is subinner in the previous terminology. Let
$$\mathcal E_g=\{h\in H^t(\T): \int_{|z|=1} z^n |h|^t \frac{|dz|}{2\pi}= \int_{\overline{\D}} z^n |g|^t d\mu(z) \text{ for all }n \ge 0\}.$$
It will follow from the proof that $f$ satisfies
 $$|f(0)|= \sup \{|h(0)|: h \in \mathcal E_g\},$$ i.e. $f$ is free outer in the previous terminology.
\begin{proof} We first show the existence of the factorization. Let $g\in P^t(\mu)$, $g \ne 0$. Note that there are really two versions of $g$, an analytic function that is defined everywhere in the open unit disc, and a $L^t(\mu)$-function that is defined a.e. $[\mu]$. It is known that the two agree a.e. $[\mu]$ in $\D$ and that $\mu|\T$ is absolutely continuous, see e.g. \cite{AlRiSuPtmu}.

Define
$$h(e^{i\theta})= \int_{\overline{\D}}\frac{1-|z|^2}{|1-\overline{z}e^{i\theta}|^2} |g(z)|^t d\mu(z).$$
Then $h\in L^1(\T)$ with $\|h\|_{L^1(\T)}=\|g\|^t_{P^t(\mu)}$.

For $u\in L^1(\T)$ write $$P[u](z)= \int_{0}^{2\pi}\frac{1-|z|^2}{|1-\overline{z}e^{i\theta}|^2} u(e^{i\theta}) \frac{d\theta}{2\pi}$$ for its Poisson integral. Note that for all $z\in \D$ and every  polynomial $p$ we have $|p(z)|^t \le P[|p|^t](z)$. Hence for all $\lambda \in \D$ we have
\begin{align*}|p(\lambda)g(\lambda)|^t &\le c \int_{\overline{\D}}|p(z)g(z)|^t d\mu \\
&\le c \int_{\overline{\D}}P[|p|^t](z)\ |g(z)|^t d\mu \\
&=c\int_0^{2\pi}|p(e^{i\theta})|^t h(e^{i\theta}) \frac{d\theta}{2\pi}.\end{align*}
Since $g\ne 0$ there must be $\lambda\in \D$ such that $g(\lambda)\ne 0$. Then Szeg\H{o}'s Theorem implies that $h$ must be $\log$-integrable; see for instance Theorem V.8.2 and the discussion in Chapter II.3 in \cite{Gamelin69}. Therefore we can define an outer function $f \in H^t$ such that $|f|^t=h$ a.e. on $\T$.
 Then $\|g\|_{P^t(\mu)} =\|f\|_{H^t(\T)}$. The measure $|f|^t \frac{|dz|}{2\pi}$ is the sweep of $|g(z)|^t d\mu(z)$ to $\T$.

 We now define an analytic function $\varphi$ in the open unit disc by $\varphi = g/f$ and we can also define $\varphi=g/f$ $[\mu]$ a.e. on $\overline{\D}$.

  The later part of  the previous string of inequalities shows that
$$\|pg\|^t_{P^t(\mu)} \le \|pf\|^t_{H^t}$$ for every polynomial $p$. Let $u\in H^t$. Since $f$ is outer, there is a sequence of polynomials such that $p_nf \to u$ in $H^t$, and we may assume that $p_n(e^{i\theta})f(e^{i\theta}) \to u(e^{i\theta})$ a.e. It is then clear that $p_n(\lambda)\to \frac{u}{f}(\lambda)$ $\mu $ a.e. Hence Fatou's lemma implies
$$\|u\varphi\|^t_{P^t(\mu)} \le \|u\|^t_{H^t}.$$

Next we show the uniqueness. Suppose $f$ and $\varphi$ satisfy the two conditions with $f(0)>0$. By the second condition we have $\|pg\|^t_{P^t(\mu)}\le \|pf\|^t_{H^t(\T)}$ for every polynomial. We take $p(z)=1+az^n$ and use $|1+az^n|^t= 1+t\mathrm{Re }\ az^n + O(|a|^2)$ and obtain
$$\int_{\overline{\D}}(1 +t \mathrm{Re }\ az^n + O(|a|^2))|g|^td\mu \le \int_{|z|=1} (1+ t \mathrm{Re }\ az^n + O(|a|^2))|f|^t \frac{|dz|}{2\pi}.$$
Now write $a=re^{i\theta}$. Then use the first condition and subtract $\|g\|^t_{P^t(\mu)}=\|f\|^t_{H^t(\T)}$ on each side, divide by $r$ and let $r\to 0$. After rearranging we obtain
$$\mathrm{Re } \ e^{i\theta} \left(\int_{\overline{\D}}z^n |g|^td\mu - \int_{|z|=1}z^n |f|^t \frac{|dz|}{2\pi} \right)\le 0 \ \text{ for all } e^{i\theta} \text{ and all }n \ge 1.$$
Thus $\int_{\overline{\D}}z^n |g|^td\mu = \int_{|z|=1}z^n |f|^t \frac{|dz|}{2\pi}$ for all integers $n$. This shows that $|f|^t$ is uniquely determined by the two conditions, hence since $f$ is outer it is uniquely determined. Then $\varphi$ is also uniquely determined.
\end{proof}

\section{Weak products}\label{SecWeakProd}
\begin{proof}[Proof of Theorem \ref{WeakProductNorm}] By scaling $h$, we may assume that $\|h\|_{\mathcal{H} \odot \mathcal{H}} = 1$.
  As mentioned in the Introduction, by the factorization result of Jury and Martin and the fact that $\HH$ has the column-row property with constant $1$ (Theorem 1.3 of \cite{JuMa2} and Theorem 1.2 of \cite{HartzColumnRow}),
  there exist $g_1,g_2 \in \mathcal{H}$ with $\|g_1\| = \|g_2\| = 1$ and
  \begin{equation*}
    h = g_1 g_2 = \Big( \frac{g_1 + g_2}{2} \Big)^2 + \Big( \frac{i (g_1 - g_2)}{2} \Big)^2.
  \end{equation*}
  Set $f_1 = (g_1 + g_2)/2$ and $f_2 = i(g_1 - g_2)/2$. By the parallelogram law,
  $\|f_1\|^2 + \|f_2\|^2 = 1$.

  By Theorem \ref{VectorFactorization} there is a  contractive column multiplier
  $
  \begin{bmatrix}
    \varphi_1 & \varphi_2
  \end{bmatrix}^T$ and a free outer function $f \in \mathcal{H}$ with $\|f\| = 1$ and
  $f_i = \varphi_i h$ for $i=1,2$. Thus,
  \begin{equation*}
    h = (\varphi_1^2 +  \varphi_2^2) f^2
    =
    \begin{bmatrix}
      \varphi_1 & \varphi_2
    \end{bmatrix}
    \begin{bmatrix}
      \varphi_1 \\ \varphi_2
    \end{bmatrix} f^2.
  \end{equation*}
  Let $\varphi = \varphi_1^2 + \varphi_2^2 \in \Mult(\mathcal{H})$,
  so that $h = \varphi f^2$.
  The column-row property with constant $1$ implies that $\|\varphi\|_{\Mult(\mathcal{H})} \le 1$.
  Observe that
  \begin{equation*}
    1 = \|h\|_{\mathcal{H} \odot \mathcal{H}}
    = \|\varphi f^2\|_{\mathcal{H} \odot \mathcal{H}}
    \le \|\varphi f\| \|f\| \le \|f\|^2 \le 1
  \end{equation*}
  so equality holds throughout.
  In particular, $\|\varphi f\| = \|f\|$, so $(\varphi,f)$ is a subinner/free outer pair,
  and $\|h\|_{\mathcal{H} \odot \mathcal{H}} = \|f\|^2$.
  Finally,
  note that $\varphi$ is in particular a contractive multiplier
  of  $\mathcal{H} \odot \mathcal{H}$, so
  \begin{equation*}
    \|h \|_{\mathcal{H} \odot \mathcal{H}} \le \|f^2\|_{\mathcal{H} \odot \mathcal{H}}
    \le \|f\|^2 = \|h\|_{\mathcal{H} \odot \mathcal{H}},
  \end{equation*}
  hence $\|h\|_{\mathcal{H} \odot \mathcal{H}} = \|f^2\|_{\mathcal{H} \odot \mathcal{H}}$.
\end{proof}
\section{Subinner functions in weighted Dirichlet spaces}\label{Sec:SubinnerDirichlet}
We explained in the Introduction that subinner functions in normalized complete Pick spaces are extreme points of the unit ball of $\Mult(\HH)$. In this Section we will prove a simple sufficient condition for a function to be subinner. As a Corollary we will see that for example in the Dirichlet space  subinner functions are quite common.

Our approach is based on an observation of Clou\^{a}tre and Davidson,  \cite[Lemma 7.4]{CD16}. It starts with the fact that an operator $A\in \HB(\HH)$ attains its norm, whenever $\|A\|_e<\|A\|$. Indeed, the hypothesis implies that $\|A^*A\|_e< \|A^*A\|=\|A\|^2$. Hence $\sigma(A^*A) \cap \{r\in \R: r >\|A^*A\|_e\}$ consists of isolated eigenvalues of $A^*A$, and since $A^*A$ is positive there must be a nonzero vector $x\in \HH$ such that $A^*Ax=\|A\|^2x$. Then $\|Ax\|=\|A\|\|x\|$. Clou\^{a}tre and Davidson further observed that for the Drury-Arveson space multipliers $\varphi$ that are multiplier norm limits of polynomials have the property that $\|\varphi\|_\infty=\|M_\varphi\|_e$. We will now show that the argument generalizes to many other spaces, and it leads to regularity conditions on $\varphi$ that will be sufficient for $\varphi$ to be subinner.

 Let $d\in \N$. A regular unitarily invariant space is a reproducing kernel Hilbert space $\mathcal{H}$ on  $\mathbb{B}_d$  whose
kernel is of the form
\begin{equation}
 \label{regularKernel} k_w(z) = \sum_{n=0}^\infty a_n \langle z,w \rangle^n,
\end{equation}
where $a_0 = 1, a_n > 0$ for all $n$ and $\frac{a_n}{a_{n+1}} \to 1$. It is well-known that a lemma of Kaluza's can be used to show that $k_w(z)$ will be a normalized complete Pick kernel, whenever $\frac{a_n}{a_{n+1}}$ is non-increasing, see e.g. \cite[Lemma 7.38]{AgMcC}.

If $f=\sum_{n=0}^\infty f_n$ is the expansion of $f$ into homogeneous polynomials of degree $n$, then the norm $\HH$ is given by $$\|f\|^2= \sum_{n=0}^\infty \frac{\|f_n\|^2_{H^2_d}}{a_n}.$$

A one-parameter family of regular unitarily invariant spaces on $\Bd$  is given by the reproducing kernels $$k^{\alpha}_w(z) = \sum_{n=0}^\infty (n+1)^{-{\alpha}} \la z,w\ra^n, \  \  \  \ {\alpha}\in \R.$$ We denote the corresponding Hilbert spaces by $\HD_{\alpha}$.

An application of Kaluza's lemma shows that $k^{\alpha}$ is a complete Pick kernel, whenever ${\alpha}\ge 0$.

A related one-parameter family of reproducing kernels on $\Bd$ is given by
$$\tilde{k}^\alpha_w(z)= \frac{1}{(1-\la z,w\ra)^{1-\alpha}}=\sum_{n=0}^\infty c_n(\alpha) \la z,w\ra^n, \alpha <1.$$ We will write $\tilde{\HD}_\alpha$ for the corresponding space of analytic functions.
The binomial series shows that
 $$c_n(\alpha) = \frac{\prod_{k=1}^n(k-\alpha)}{n!}, \ \ n\in \N.$$ Thus, another application of Kaluza's Lemma implies that $\tilde{k}^\alpha$ is a normalized complete Pick kernel, whenever $0\le \alpha<1$. Furthermore, by Gauss's formula for the Gamma function we have for all $\alpha <1$ \begin{equation} \label{CoeffRatio} \begin{aligned}[m]0<\Gamma(1-\alpha)&=(-\alpha)\Gamma(-\alpha) \\ &= \lim_{n\to \infty} \frac{ n! (n+1)^{-\alpha}}{(1-\alpha)(2-\alpha) \cdots (n-\alpha)}\\&=\lim_{n\to \infty} \frac{(n+1)^{-\alpha}}{c_n(\alpha)},\end{aligned}\end{equation} see e.g. \cite{ConwayFOCV} page 178. This implies that for all $\alpha <1$ we have $\HD_\alpha= \tilde{\HD}_\alpha$ with equivalence of norms.

It is well-known that for ${\alpha}<-d+1$, the norm of $\tilde{\HD}_{\alpha}$ is  a weighted Bergman norm,
$$\|f\|^2_{\tilde{\HD}_\alpha} =c_\alpha \int_{\Bd} |f(z)|^2 (1-|z|^2)^{-d-{\alpha}} dV(z).$$ Here $V$ is used to denote Lebesgue measure on $\Bd$ and $c_\alpha$ is a normalization constant chosen so that $\|1\|^2_{\tilde{\HD}_\alpha}=1$, see e.g. \cite{ZhuBallBook}, Theorem 2.2.

Particular examples are the Drury-Arveson space $H^2_d=\HD_0=\tilde{\HD}_0$, and for $d=1$  we have $D=\HD_1$, the Dirichlet space. We also have $\tilde{\HD}_{1-d}=H^2(\partial \Bd)$ and $\tilde{\HD}_{-d}=L^2_a(dV)$, the unweighted Bergman space, all with equality of norms.

\

Now let $\HH$ be a regular unitarily invariant space. The polynomials are multipliers of $\mathcal{H}$, so we may define $A(\mathcal{H})$ to be the norm
closure of the polynomials inside of $\Mult(\mathcal{H})$. Furthermore, we write $C^*(A(\HH))$ for the $C^*$-algebra generated in $\HB(\HH)$ by $\{M_\varphi: \varphi\in A(\HH)\}$, and we use $K(\HH)$ to denote the compact operators in $\HB(\HH)$.

For us the  significance of being a regular unitarily invariant is that,
just like in the Hardy space,
we have a short exact sequence of $C^*$-algebras
\begin{equation}
  \label{eqn:short_exact}
  0 \to K(\mathcal{H}) \to C^*(A(\mathcal{H})) \to C( \partial \mathbb{B}_d) \to 0,
\end{equation}
where the second map is the inclusion and the third map is the unital $*$-homomorphism that takes a function in $A(\mathcal{H})$ to its boundary values;
see \cite[Theorem 4.6]{GHX04}. It particular, this implies that $M_\varphi$ is essentially normal and the essential norm satisfies $\|M_\varphi\|_e=\|\varphi\|_\infty$, whenever $\varphi \in A(\mathcal{H})$. Thus, by the discussion at the beginning of the section, a multiplier $\varphi\in A(\HH)$ will be subinner, whenever $\|\varphi\|_\infty <\|\varphi\|_{\MuH}=1$.

In order to be able to consider multipliers that may not be in $A(\HH)$ we use a second condition. We say that the one-function Corona Theorem holds for $\MuH$, if whenever $\varphi \in \MuH$ with $|\varphi(z)|\ge c>0$ for all $z\in \Bd$, then $1/\varphi\in \MuH$.  It is known that  the one-function Corona Theorem holds for $\Mult(\HD_{\alpha})$ for all $ {\alpha}\in \R$. For ${\alpha}\le 1-d$ this is trivial since $\Mult(\HD_\alpha)=H^\infty(\Bd)$ in this case, and for ${\alpha}\ge 0$ it is a consequence of the full Corona Theorem, see \cite{CosteaSawWick}. For a proof of just the one-function Corona Theorem see  \cite{FangXia} (${\alpha}=0$) and \cite{RichterSunkes},  Theorem 5.4 for all ${\alpha}<1$. By use of Corollary 3.4 of \cite{AHMRRadBesov} one can extend the proof of \cite{RichterSunkes}  to cover all $\alpha \in \R$.
For an example of a regular unitarily invariant Pick kernel that does not satisfy the one-function Corona Theorem see \cite{AHMcCR_Smirnov}.

If $\MuH$ satisfies the one-function Corona Theorem, then one easily shows that $\sigma(M_\varphi)=\overline{\varphi(\Bd)}$ for every $\varphi\in \MuH$. Thus, if $M_\varphi$ is essentially normal, then its
essential norm equals its essential spectral radius which is dominated by the spectral radius. For such multipliers we  thus obtain $\|M_\varphi\|_e = r_e(M_\varphi) \le r(M_\varphi)=\|\varphi\|_\infty$.

Hence we have proved the following Proposition.

\begin{prop}
  \label{prop:subinner_mult_norm}
  Let $\mathcal{H}$ be a regular unitarily invariant space, let $\varphi \in \MuH$
  and suppose that $\|\varphi\|_\infty < \|\varphi\|_{\Mult(\mathcal{H})} = 1$.

If either  \begin{itemize} \item $\varphi\in A(\HH)$ or \item $\MuH$ satisfies the one-function Corona theorem and $M_\varphi$ is essentially normal,\end{itemize} then  $\varphi$ is subinner, i.e.\ $M_{\varphi}$ attains its norm.
\end{prop}

The following result is well-known; see \cite[Example 1, p. 99]{Shields74} for $d=1$ and \cite[Theorem 4.1]{OF06} for $d \ge 1$.
Since we do not have a convenient reference for the proof, we sketch an argument that uses reproducing kernels.

\begin{lemma}\label{DalphaAlgebra}  Let $d \in \mathbb{N}$.   If $\alpha > 1$, then $A(\HD_\alpha)=\Mult(\mathcal{D}_\alpha) = \mathcal{D}_\alpha$ with equivalence of norms.
\end{lemma}

\begin{proof}
  It suffices to show $\mathcal{D}_\alpha = \Mult(\mathcal{D}_\alpha)$. Since polynomials are dense in $\mathcal{D}_\alpha$,
  it then follows that $A(\mathcal{D}_\alpha) = \Mult(\mathcal{D}_\alpha)$.

  Clearly, there is a contractive inclusion $\Mult(\mathcal{D}_a) \subset \mathcal{D}_a$.
  To show a continuous inclusion in the other direction, let $a_n = (n+1)^{-\alpha}$.
  The argument in \cite[Example 1, p. 99]{Shields74} shows that there exists $C \ge 0$ such that
  \begin{equation*}
    \sum_{k=0}^n \frac{a_k a_{n-k}}{a_n} \le C^2 \quad \text{ for all }  n \in \mathbb{N},
  \end{equation*}
  from which it follows that
  \begin{equation}
    \label{eqn:kernel_ineq}
    C^2 k^\alpha - (k^\alpha)^2 \ge 0,
  \end{equation}
  meaning that this function is positive definite.
  Now, if $f \in \mathcal{D}_\alpha$ with $\|f\| \le 1$, then
  \begin{equation*}
    k^\alpha_w(z) - f(z) \overline{f(w)} \ge 0.
  \end{equation*}
  Multplying by $k^\alpha$, using the Schur product theorem and \eqref{eqn:kernel_ineq}, it follows that
  \begin{equation*}
    k^\alpha_w(z) ( C^2 - f(z) \overline{f(w)}) \ge 0,
  \end{equation*}
  so $f$ is a multiplier of multiplier norm at most $C$.
\end{proof}

We obtain the following consequence for subinner functions in $\mathcal{D}_\alpha$ for $\alpha > 1$.

\begin{theorem}\label{subinnerDa_algebra} If ${\alpha}>1$, then $\varphi\in \Mult(\HD_{\alpha})$ is subinner, if and only if $\|\varphi\|_{\Mult(\HD_{\alpha})}=1$.\end{theorem}
\begin{proof} It is clear that subinner functions must have multiplier norm equal to 1. Let $\varphi\in \Mult(\HD_{\alpha})$  with $\|\varphi\|_{\Mult(\HD_{\alpha})}=1$. If $\varphi$ is constant, then it clearly must be subinner, thus we assume that $\varphi$ is not constant.  Since  ${\alpha}>1$ we note that $\sup_{|z|<1}\|k_z^{\alpha}\|<\infty$, and this implies that  every function $f\in \HD_{\alpha}$ is bounded. Hence Proposition 4.7 of \cite{Serra} (see also Lemma \ref{lem:liminf} below) implies $\|\varphi\|_\infty < \|\varphi\|_{\Mult(\HD_{\alpha})}$. Thus, by Proposition \ref{prop:subinner_mult_norm} and Lemma \ref{DalphaAlgebra} we conclude that $\varphi$ is subinner.    \end{proof}

  \begin{theorem} Let $\HH$ be a reproducing kernel Hilbert space with a normalized complete Pick kernel. If $\HH$ contains an unbounded function, then there is $\psi\in \MuH$ such that $\|\psi\|_{\Mult(\HH)}=1$, but $\psi$ is not subinner.
  \end{theorem}   Serra had already proved that if a reproducing kernel Hilbert space with normalized complete Pick kernel contains an unbounded function, then
 there are functions in $\Mult(\HH)$ such that $\|\varphi\|_\infty=\|\varphi\|_{\Mult(\HH)}$, \cite{Serra}. The proof of our theorem provides an alternate approach to that result.
  \begin{proof} Let $f\in \HH$, $\|f\|=1$ and assume that $f$ is unbounded. By Theorem 1.1 of \cite{AHMcCR_Factor} there are $\varphi, \psi\in \MuH$ such that $\|\psi h\|^2+\|\varphi h\|^2\le \|h\|^2$ for all $h\in \HH$ and
  $f=\varphi/(1-\psi)$. Since $\varphi$ is bounded, we see that the unboundedness of $f$ implies that there are points $z_n$ such that $\psi(z_n)\to 1$. Hence $1\le \|\psi\|_\infty \le \|\psi\|_{\MuH}\le 1$. Thus, $\|\psi\|_{\MuH}= 1$ and it is not column extreme, hence by Corollary \ref{subinner has free outer} it cannot be subinner.
  \end{proof}

 In order to obtain a sufficient condition for being subinner we start by proving a lemma.

\begin{lemma}\label{lem:liminf} Let $\HH_k$ be a reproducing kernel Hilbert space on $\Bd$ with reproducing kernel $k$, which is normalized at 0. If $\varphi\in \Mult(\HH_k)$ is not constant, if $\|\varphi\|_\infty=1$, and if
$$\liminf_{|z|\to 1}(1-|\varphi(z)|^2)\|k_z\|^2 < \frac{1-|\varphi(0)|}{1+|\varphi(0)|},$$ then  $\|\varphi\|_\infty < \|\varphi\|_{\Mult(\HH_k)}$.
\end{lemma}
Note that the lemma also implies  Proposition 4.7 of \cite{Serra} that was used in the  proof of Theorem \ref{subinnerDa_algebra}.
\begin{proof} By the hypothesis we may choose $z_n \in \Bd$ such that $|z_n|\to 1$ and $\lim_{n \to \infty}(1-|\varphi(z_n)|^2)\|k_{z_n}\|^2 < \frac{1-|\varphi(0)|}{1+|\varphi(0)|}$.
 Note that since $\varphi$ is not constant, we have $|\varphi(0)|<1$.

We know that $\|\varphi\|_\infty \le \|\varphi\|_{\Mult(\HH_k)}$, so suppose that $\|\varphi\|_{\Mult(\HH_k)}=1$. Then for any $z, w\in \Bd$ the corresponding $2\times 2$ Pick matrix is positive semidefinite. For $w=0$, $z=z_n$ this is easily seen to be equivalent to

$$(1-|\varphi(z_n)|^2)(1-|\varphi(0)|^2)\|k_{z_n}\|^2\ge   |1-\overline{\varphi(0)}\varphi(z_n)|^2.$$

This implies
$${(1-|\varphi(z_n)|^2)}\|k_{z_n}\|^2 \ge  \frac{|1-\overline{\varphi(0)}\varphi(z_n)|^2}{1-|\varphi(0)|^2} \ge \frac{1-|\varphi(0)|}{1+|\varphi(0)|}.$$ This contradicts the hypothesis, hence we conclude $\|\varphi\|_\infty < \|\varphi\|_{\Mult(\HH_k)}$.
\end{proof}
We note that the proof of the Lemma shows that the hypothesis implies that what has been called the  ``2-point multiplier norm of $\varphi$'' is strictly larger than $\|\varphi\|_\infty=1$, see \cite{AHMRSubhomog}. For the Drury-Arveson space it turns out that the 2-point multiplier norm equals the $H^\infty$-norm, hence the lemma cannot be used to construct examples of subinner functions in the Drury-Arveson space, see \cite[Lemma 3.3]{HaRiShalit}.

We will apply the lemma to the spaces $\HD_\alpha$, $0<\alpha \le 1$, and to $\tilde{\HD}_\alpha$, $0<\alpha<1$.

In order to apply Proposition \ref{prop:subinner_mult_norm} and Lemma \ref{lem:liminf} we need to know when multiplication operators are essentially normal. For the Dirichlet space $D$ of the unit disc Lech provided a necessary and sufficient condition for $M_\varphi$ to be essentially normal, see \cite{Lech}. He proved that if $\varphi\in \Mult(D)$, then $M_\varphi$ acts as an essentially normal operator on $D$, if and only if $M_{\varphi'}: D \to L^2_a$ is compact. Then one can for example obtain a simple sufficient condition by calculating the Hilbert-Schmidt norm of $M_{\varphi'}$, see the paper by Brown and Shields \cite{BrownShields}, top of page 303. If $d>1$,
then not even the constant function $1$ yields a Hilbert-Schmidt multiplication operator $\mathcal{D}_1 \to \mathcal{D}_{-1}$, but a condition analogous to the Brown-Shields condition turns out to be sufficient for an essentially normal multiplier.

If $0<\alpha<1$, then the following lemma together with equation (\ref{CoeffRatio}) will show that for $f\in \Mult(\HD_\alpha)$ the multiplication operator $M_f$ acts essentially normally on $\HD_\alpha$, if and only if it is essentially normal on $\tilde{\HD}_\alpha$.
In the following we write $(M_\varphi, \mathcal{H})$ to denote multiplication by $\varphi$ as it acts on $\mathcal{H}$.

\begin{lemma}  \label{similarity} Let $d\in \N$ and let $a=\{a_n\}$ and $b=\{b_n\}$ be sequences of positive real numbers that determine regular unitarily invariant spaces $\HH_a$ and $\HH_b$ on $\Bd$ with kernels $k^a$ and $k^b$ defined by (\ref{regularKernel}). If $$\lim_{n\to \infty} \frac{b_n}{a_n}=c>0,$$ then $\HH_a=\HH_b$ with equivalent norms and if $\varphi\in \Mult(\HH_a)=\Mult(\HH_b)$, then $(M_\varphi, \HH_a)$ is essentially normal, if and only if $(M_\varphi, \HH_b)$ is essentially normal.
\end{lemma}
\begin{proof} It is clear that the hypothesis implies that the spaces coincide and the norms are equivalent. Write $T_x=(M_\varphi, \HH_x)$ for $x=a$ and $x=b$.  Let $S\in \HB(\HH_a,\HH_b)$ be defined by $Sf= \sum_{n=0}^\infty \frac{b_n}{a_n}f_n$, where $f = \sum_{n=0}^\infty f_n$ is the homogeneous expansion of $f$. Then the hypothesis implies that $S$ is invertible and $Sk^a_w=k^b_w$ holds for all $w\in \Bd$.

Then we have $$T_b^*Sk_w^a= \overline{\varphi(w)}k^b_w= ST_a^*k^a_w$$ for all $w\in \Bd$. This implies that $T_b^*S= ST_a^*$. Note that
$Uf= \sum_{n=0}^\infty \frac{\sqrt{a_n}}{\sqrt{b_n}} f_n$ is a Hilbert space isomorphism
$\HH_b\to \HH_a$ and we have
$(UT_b^*U^*)(US)= US T_a^*.$
 The lemma now follows from the fact that $US\in \HB(\HH_a)$ is of the form $US=\sqrt{c}I+K$ for some compact operator $K$.\end{proof}

\

If $f \in \Hol(\Bd)$, then we write  $f=\sum_{n=0}^\infty f_n$ for the expansion of $f$ into homogeneous polynomials. We define the spaces
$$\HD_\alpha^0=\{f\in \Hol(\Bd): \|f\|^2_{\HD_\alpha^0} = \sum_{n=0}^\infty (n+1)^\alpha \log(n+2) \|f_n\|^2_{H^2_d}<\infty\}.$$ We will only need the cases where $\alpha \in \{\pm 1\}$. Let $R = \sum_{i=1}^d z_i \frac{\partial}{\partial z_i}$ be the radial derivative operator, $Rf=\sum_{n=1}^\infty nf_n$.

\begin{lemma}\label{MultiplierEstimate} (a) ($\alpha=1$) There is $c>0$ such that $$\|gRf\|_{\HD_{-1}}\le c \|f\|_{\HD_1^0} \|g\|_{\HD_{1}}$$ for all $g \in {\HD_{1}}, f\in \HD_1^0$.

 (b) If $\alpha <1$, then there is $c>0$ such that
$$\|gRf\|_{\HD_{\alpha-2}} \le c \|f\|_{\HD_{1}}\|g\|_{\HD_\alpha}$$ for all $g\in \HD_\alpha$, $f \in \HD_{1}.$
\end{lemma}
\begin{proof} (a) Note that $\HD_{-1}^0$ has reproducing kernel $$k^{(-1,0)}_w(z) =\sum_{n=0}^\infty \frac{n+1}{\log(n+2)}\la z,w\ra^n,$$ and, of course, for $\alpha\in \{\pm1\}$ the space $\HD_\alpha$ has reproducing kernel $k^\alpha_w(z)= \sum_{n=0}^\infty (n+1)^{-\alpha}\la z,w\ra^n$.

 Let $f\in \HD_1^0$ with $\|f\|_{\HD_1^0}\le 1$. Then $\|Rf\|_{\HD_{-1}^0}\le 1$ and that is equivalent to
 $$k^{(-1,0)}_w(z)-Rf(z) \overline{Rf(w)}\ge 0.$$ We now multiply that inequality by $k^1_w(z)$ and apply the Schur product theorem to conclude that
 $$k^{(-1,0)}_w(z)k^1_w(z) -Rf(z) \overline{Rf(w)}k^1_w(z)\ge 0.$$ Part (a) of the Lemma will follow once we show that there is $c>0$ such that $$c^2k^{-1}_w(z)-k^{(-1,0)}_w(z)k^1_w(z) \ge 0,$$ since adding both inequalites gives that $R f$ is a multiplier from $\mathcal{D}_1$ to $\mathcal{D}_{-1}$ of norm at most $c$. All three reproducing kernels are power series in $\la z,w\ra$, hence the last inequality can be verified by showing that all coeffcients in the resulting power series are $\ge 0$. Hence we have to show that there is $c>0$ such that for each $n \in \N$
$$\sum_{k=0}^n \frac{1}{k+1} \frac{n-k+1}{\log(n-k+2)} \le c^2(n+1).$$ Thus, we calculate
\begin{align*}\sum_{k=0}^n &\frac{1}{k+1} \frac{n-k+1}{\log(n-k+2)}\\ &= \sum_{0\le k \le n/2} \frac{1}{k+1} \frac{n-k+1}{\log(n-k+2)}+\sum_{n/2<k\le n} \frac{1}{k+1} \frac{n-k+1}{\log(n-k+2)}\\
&\le \sum_{0\le k \le n/2} \frac{1}{k+1} \frac{n+1}{\log(\frac{n}{2}+2)}+\sum_{n/2<k\le n} \frac{1}{\log(2)}\\
&\le (1+\log(\tfrac{n}{2}) ) \frac{n+1}{\log(\frac{n}{2}+2)} + \frac{n}{2\log 2}\\
&\le c^2(n+1).
\end{align*}
This proves (a).

(b) follows similarly noting that for $\alpha < 1$ we can substitute the spaces $\tilde{\HD}_\alpha$ and then use that the reproducing kernels satisfy the identity
  $\tilde{k}^\alpha \tilde{k}^{-1}= \tilde{k}^{\alpha-2}$.
\end{proof}

We now obtain a sufficient condition for essential normality of multiplication operators.
\begin{theorem} \label{essentiallyNormal} Let $d\in \N$, and let $\varphi \in H^\infty(\Bd)$.

(a) If $\varphi \in \HD_1^0$, then $\varphi\in \Mult(\HD_1)$ and  $(M_\varphi, \HD_1)$ is essentially normal.

(b) If $\varphi\in \HD_1$, then for each $\alpha < 1$, we have  $\varphi\in \Mult(\HD_\alpha)= \Mult(\tilde{\HD}_\alpha)$ and  $(M_\varphi, \HD_\alpha)$ and $(M_\varphi, \tilde{\HD}_\alpha)$ are essentially normal.
 \end{theorem}

\begin{proof} We start by discussing the situation when $\alpha <1-d$. Then $\tilde{\HD}_\alpha$ is a weighted Bergman space contained in $L^2(c_\alpha(1-|z|^2)^{-d-\alpha}dV)$ and for $\varphi\in H^\infty(\Bd)$ we have $M_\varphi^*M_\varphi-M_\varphi M_{\varphi}^*= H_{\overline{\varphi}}^* H_{\overline{\varphi}},$ where $H_{\overline{\varphi}}= (I-P_\alpha)M{\overline{\varphi}}|\tilde{\HD}_\alpha$ is the big Hankel operator with symbol $\overline{\varphi}$. Here we used $P_\alpha$ to denote the orthogonal projection of $L^2(c_\alpha(1-|z|^2)^{-d-\alpha}dV)$ onto $\tilde{\HD}_\alpha$. It is known that $H_{\overline{\varphi}}$ is compact, if and only if $\varphi$ is in the little Bloch space, i.e. if $\lim_{|z|\to 1} (1-|z|^2)|R\varphi(z)|=0$, see \cite[Corollary 24]{Zhu92}.

Next we note that for $f\in \HD_1$ and $z\in \Bd$ we have $(1-|z|^2)|Rf(z)|\le \|Rf\|_{\HD_{-1}}\le \|f\|_{\HD_1}$ by a standard estimate for point evaluations. Hence if $\varphi\in \HD_1$, and if $\varphi_r(z)=\varphi(rz)$ for $0<r<1$, then
\begin{align*}\limsup_{|z|\to 1} (1-|z|^2)|R\varphi(z)| &=\limsup_{|z|\to 1} (1-|z|^2)|R(\varphi-\varphi_r)(z)| \\ &\le \|\varphi-\varphi_r\|_{\HD_1} \to 0 \text{ as }r\to 1.\end{align*} Hence $\varphi$ is in the little Bloch space, and thus the hypotheses in (a) and (b) both imply that $(M_\varphi, \tilde{\HD}_\alpha)$ is essentially normal for all $\alpha <1-d$. Then by Lemma \ref{similarity} the same is true for  $(M_\varphi, {\HD}_\alpha)$,  $\alpha <1-d$, and from now on we will just consider the case of $(M_\varphi, {\HD}_\alpha)$.

Fix $\alpha \le 1$, and  observe that for all $t\in \R$ the operator $R+I$ is unitary from $\HD_t$ onto $\HD_{t-2}$.  Furthermore, the product rule implies that $(R+I)(M_\varphi,\HD_t) =(M_\varphi,\HD_{t-2}) (R+I) +M^t_{R\varphi}$, where $M^t_{R\varphi}: \HD_t\to \HD_{t-2}, f \to (R\varphi)f$. Thus, an easy induction argument implies that both parts (a) and (b) of the  lemma follow, if we show that the hypotheses imply that $M^t_{R\varphi}$ is compact for all $t \le \alpha$.

(a) Let $\varphi\in \HD_1^0$. Lemma \ref{MultiplierEstimate} implies that $M^1_{R\varphi}\in \HB(\HD_1,\HD_{-1})$. Note that for $r<1$ the operator $R\varphi_r$ is a multiplier of $\HD_1$ into itself. Since $d<\infty$ the inclusion $\HD_1\to \HD_{-1}$ is compact, hence $M^1_{R\varphi_r}$ is compact.
By Lemma \ref{MultiplierEstimate} again there is $c>0$ such that for all $0\le r<1$ we have $\|M^1_{R\varphi}-M^1_{R\varphi_r}\|_{\HB(\HD_1,\HD_{-1})}\le c\|\varphi-\varphi_r\|_{\HD_1^0}$. It is clear that $\|\varphi-\varphi_r\|_{\HD_1^0}\to 0$, hence $M^1_{R\varphi}$ is a norm limit of compact operators and hence it is compact as well.

In order to finish the proof of (a) and to prove (b) it suffices  to show that $M^t_{R\varphi}$ is compact for all $t<1$, whenever $\varphi\in \HD_1$. The proof is similar to the one just given, except that we substitute the use of Lemma \ref{MultiplierEstimate} (b), where we had used part (a) of that same lemma.
\end{proof}

If $0<\beta \le 1$, then we write  $\mathrm{Lip} \beta$ for the holomorphic functions on $\Bd$ that satisfy a Lipschitz condition of order $\beta$, and we write $\mathrm{lip}\beta$ for those $f\in \mathrm{Lip}\beta$ that satisfy $$\limsup_{|z-w|\to 0} \frac{|f(z)-f(w)|}{|z-w|^\beta}=0.$$
We also let $\mathbb{A}(\mathbb{B}_d)$ denote the ball algebra, i.e.\ the algebra of all holomorphic functions on $\mathbb{B}_d$ that extend
to be continuous on $\overline{\mathbb{B}_d}$.
\begin{theorem} \label{DiriSubinner} Let $d\in \N$.

(a) If  $0 < \alpha < 1$ and $\varphi\in \mathrm{lip}(1-\alpha)\cap \HD_1$, then  $\varphi \in  \Mult(\mathcal{D}_\alpha)$ and $\varphi$ is subinner if and only if $\|\varphi\|_{\Mult(\mathcal{D}_\alpha)} = 1$.

(b) ($\alpha=1$) If $\varphi\in \mathbb{A}(\Bd)\cap \HD_1^0$ and such that
 $\lim_{r\to 1} \log\frac{1}{1-r} |\varphi(z)-\varphi(rz)|=0$ for all $z\in \dB$.
Then $\varphi \in  \Mult(\mathcal{D}_1)$ and $\varphi$ is subinner if and only if $\|\varphi\|_{\Mult(\mathcal{D}_1)} = 1$.
\end{theorem}
By Lemma \ref{similarity} it is clear that a similar theorem holds for $\tilde{\HD}_\alpha$, $0 < \alpha < 1$. We also note that for $\alpha=0$ (the Drury-Arveson space) no direct analogue of this theorem can hold. Indeed, the function $\varphi(z)=\frac{1+z_1}{2}$ is a polynomial of multiplier norm 1 that is not subinner in $\HD_0=H^2_d$.
\begin{proof}  (a) The Lipschitz condition implies that $\varphi\in \mathbb{A}(\Bd)$, hence Theorem \ref{essentiallyNormal} (b) implies that $\varphi \in  \Mult(\mathcal{D}_\alpha)$ and $M_\varphi$ is essentially normal.

 Suppose $\|\varphi\|_{\Mult(\mathcal{D}_\alpha)} = 1$.   The assertion is trivial if $\varphi$ is constant, so we assume that $\varphi$ is non-constant. Then by Proposition \ref{prop:subinner_mult_norm} and the one function Corona theorem for $\mathcal{D}_\alpha$, it suffices to show that $\|\varphi\|_\infty <1.$

  The regularity assumption on $\varphi$ implies that the function  $\psi=\varphi/\|\varphi\|_\infty$ extends to be continuous on $\overline{\Bd}$. Suppose that $|\psi|$ attains its maximum 1 at $w\in \dB$, then for $z=rw\in \Bd$ we have
  \begin{align*} (1-|\psi(z)|^2)\|k_z\|^2 &\le 2|\psi(w)-\psi(rw)|\|k_{rw}\|^2\\
  &\lesssim \frac{|\varphi(w)-\varphi(rw)|}{|w -rw|^{1-\alpha}}\to 0 \ \ \text{ as }r\to 1,\end{align*}
  where we used that $\|k_{r w}\|^2 \approx 1/(1 - r)^{1 - \alpha}$.
  Thus, Lemma \ref{lem:liminf} implies that $\|\varphi\|_\infty<\|\varphi\|_{\Mult(\HD_{\alpha})}=1$.

  The proof of (b) is similar, it relies on Lemma \ref{essentiallyNormal} (a).
\end{proof}

Note that in conjunction with Example 10.3 the theorem shows that for every $|\lambda|\ge \sqrt{2}$ the function $(z-\lambda)/\|z-\lambda\|_{\Mult(D)}$ is both subinner and free outer in the Dirichlet space $D$.

\bibliography{FreeOuterBiblio}
\end{document}